\documentclass[10pt, a4paper]{article}

\usepackage{natbib}
 \bibpunct[, ]{(}{)}{,}{a}{}{,}%
 \def\bibfont{\small}%
 \def\bibsep{\smallskipamount}
 \def\bibhang{24pt}%
\usepackage[colorlinks=true,bookmarks=false,urlcolor=blue,
     citecolor=blue,linkcolor=blue,bookmarksopen=false,draft=false]{hyperref}

\usepackage{amsbsy, amsfonts, amsgen, amsmath, amsopn, amssymb, amstext,
amsthm, amsxtra, bezier, color, enumerate, graphicx, latexsym, verbatim,
pictexwd, supertabular, url, dsfont,leftidx, mathrsfs, appendix,setspace}
\usepackage{algorithm2e}
\usepackage{graphicx}
\usepackage[labelfont=bf]{caption}

\usepackage{etex}
\usepackage{slashbox,pict2e}
\usepackage{makecell}
\usepackage{longtable}
\usepackage{tabularx}
\usepackage{booktabs}
\usepackage{etex}
\usepackage{multirow}
\usepackage{array}

\newtheorem{theorem}{Theorem}
\newtheorem{lemma}{Lemma}
\newtheorem{proposition}{Proposition}

\usepackage{longtable}
\usepackage{tabularx}

\begin{document}

\title{Dynamic Allocation Problems in Loss Network Systems \\with Advanced Reservation}

\author{{\normalsize Retsef Levi$^\ast$, Cong Shi$^\dagger$}\\
{\scriptsize $^\ast$ Sloan School of Management, Massachusetts Institute of Technology, Cambridge, MA 02139, retsef@umich.edu}\\
{\scriptsize $^\dagger$ Industrial and Operations Engineering, University of Michigan, Ann Arbor, MI 48109, shicong@umich.edu}\\
}
\date{}

\maketitle

\begin{abstract}

We consider a class of well-known dynamic resource allocation models in loss network systems with advanced reservation. The most important performance measure in any loss network system is to compute its blocking probability, i.e., the probability of an arriving customer in equilibrium finds a fully utilized system (thereby getting rejected by the system). In this paper, we derive upper bounds on the asymptotic blocking probabilities for such systems in high-volume regimes. There have been relatively few results on loss network systems with advanced reservation due to its inherent complexity. The theoretical results find applications in a wide class of revenue management problems in systems with reusable resources and advanced reservation, e.g., hotel room, car rental and workforce management. We propose a simple control policy called the \emph{improved class selection policy} (ICSP) based on solving a continuous knapsack problem, similar in spirit to the one proposed in \cite{levi}. Using our results derived for loss network systems with advanced reservation, we show the ICSP performs asymptotically near-optimal in high-volume regimes. 

\flushleft{\em History: Submitted May 2015}
\end{abstract}

\section{Introduction}
In this paper, we consider a class of well-known dynamic allocation problems in loss network systems with advanced reservation. The stochastic system consists of a homogeneous pool of resources with a known fixed capacity. Requests for using this pool of resources belong to a diverse set of customer classes, differing in the arrival rate, the advanced reservation duration (the time between a request and its start of service), the service duration, and the willingness to pay. In particular, each customer requests a service interval, specifying a start time (potentially in the future) and an end time. Customers cannot be delayed and their requests must be addressed instantaneously upon arrival. That is, each request must either be reserved into the system for service and assigned an appropriate resource unit for its requested service interval or rejected (lost) at the instant it arrives. An admitted request occupies the allocated resource unit for its service duration and releases the resource unit after the service is completed. The released resource unit can be used to serve other customers. The goal is to design an admission control policy that optimizes the expected long-run average revenue rate. 

Dynamic allocation problems in loss network systems with advanced reservation are motivated by both traditional and emerging application domains, such as \emph{hotel room}, \emph{car rental} and \emph{workforce management}. For instance, in hotel industries, customers make requests to book a room in the future for a specified number of days, which is often referred to as advanced reservation. Rooms are allocated to customers based on their requests, and after one customer used a room it becomes available to serve other customers. One of the major issues in these systems is how to manage capacitated pool of reusable resources over time in a dynamic environment with many uncertainties, in order to maximize the long-run expected revenue.

\subsection{Main results and contributions of this paper}
Firstly, we provide the first asymptotic analysis on the blocking probabilities in loss network systems with discrete advanced reservation and service distributions. The customers arrive to a service system (with capacity $C$) according to $M$-class Poisson processes and are being served as long as there is available reservation capacity. Customers who find a fully utilized system are rejected and lost (see, e.g., \cite{kelly}). Let the traffic intensity $\rho = \sum_{k=1}^{M}\lambda_{k}\mu_{k}$ where $\lambda_{k}$ and $\mu_{k}$ are the Poisson arrival rate and the mean service time of class $k$, respectively. We derive explicit upper bounds on the steady-state \emph{blocking probability}, i.e., the probability that a random customer at steady-state will find a fully utilized system (and is therefore rejected), and analyze them asymptotically in high-volume regimes. 

\begin{theorem}
\label{mainresult}
Consider a loss network system with discrete advanced reservation and service distributions. Let $C$ and $\rho$ be the capacity and the traffic intensity of the system, respectively.
\begin{enumerate}[(a)]
\item Under the high-volume regime where $C=\rho \rightarrow \infty$, the blocking probabilities have an asymptotic upper bound of $1/2$. 
\item Under the high-volume regime where $C=(1+\epsilon) \rho  \rightarrow \infty$ for any small positive $\epsilon$, the blocking probabilities are asymptotically zero.
\end{enumerate}
\end{theorem}

To the best of our knowledge, there have been very few successful attempts to characterize the blocking probabilities for loss network models with advanced reservation (see, e.g., \cite{coffman} and \cite{lura} that studied several special cases). One of the major difficulties in models with advanced reservation is the fact that a randomly arriving customer effectively observes a nonhomogeneous Poisson process that is induced by the already reserved service intervals. Moreover, analyzing the blocking probability of an arriving customer requires considering the entire requested service interval instead of the instantaneous load of the system. Analyzing the load over an interval immediately introduces correlation that is challenging to analyze. The upper bound on the blocking probability is obtained by considering identical systems with infinite capacity, where all customers are admitted ($M/G/\infty$ systems with advanced reservation). The probability of having more than $C$ customers reserved in the infinite capacity system provides an upper bound on the blocking probability in the original system; we call this the \emph{virtual blocking probability}. Through an innovative reduction to an asymmetric random walk, we obtain an exact analytical expression for this virtual blocking probability and then analyze it asymptotically. The analytical techniques used in developing our results are completely different from the ones used in \cite{levi}.

Secondly, the theoretical results in loss network systems with advanced reservation find interesting applications in revenue management of reusable resources with advanced reservation, e.g., hotel room, car rental and workforce management. A simple control policy called the \emph{improved class selection policy} (ICSP) is proposed based on solving a continuous knapsack problem, similar in spirit to \cite{levi}. Using our results derived for loss network systems with advanced reservation, we show that the ICSP performs asymptotically near-optimal in high-volume regimes.

\begin{theorem}
\label{mmthm}
Consider the revenue management model with a single pool of capacitated reusable resources and advanced reservation under the ICSP.
Under the high-volume regime where $C=(1+\epsilon) \rho  \rightarrow \infty$ for any small positive $\epsilon$, the ICSP is guaranteed to obtain at least $(1-\epsilon)$ of the optimal long-run expected revenue.
\end{theorem}

We also carried out an extensive numerical study on the ICSP in comparison with the optimal stochastic dynamic programming solutions. Our results show that our policy performs within a few percentages of the optimal for a large set of parameters (even in light-traffic). There have been relatively few results on loss network systems with advanced reservation, and we believe that the approaches developed in this paper will be applicable in other applications domains in operations management.

\subsection{Relevant Literature}
Loss network systems without advanced reservation are well-known; they were introduced over four decades ago and have been studied extensively, primarily in the context of communication networks (e.g., the survey paper by \cite{kelly}) and recently other application domains. Two of the major issues in the literature on loss networks have been the \emph{study and design of heuristics} for admission control (e.g., \cite{miller}, \cite{ross}, \cite{key}, \cite{kelly}, \cite{hunt}, \cite{puhal}, \cite{fan}), and the development of \emph{approximations and bounds} as well as \emph{sensitivity analysis} of blocking probabilities with respect to input parameters and resource capacities (e.g., \cite{erlang}, \cite{sevastyanov}, \cite{kaufman}, \cite{burman}, \cite{whitt}, \cite{kelly}, \cite{ross2}, \cite{zach}, \cite{louth}, \cite{kumar} and \cite{adelman2}). However, there have been relatively few successful attempts to characterize the blocking probabilities for the loss network systems with advanced reservation. In particular, all the results mentioned above do not carry through. \cite{coffman} derived explicit formulas for the limiting blocking probabilities in several special cases, for instance, in a setting where the reservation distribution is uniform and all requested intervals have unit length. They extended the result to more general reservation distributions by relating the problem to an on-line interval packing problem. \cite{lura} studied the asymptotic blocking probabilities when the capacity of the system approaches infinity with sub-exponential resource requirements. Some papers are devoted to study the transient behavior or approximations of blocking probabilities for the $M_{t}/G/\infty$ queue as well as $M_{t}/G/C/C$ loss systems (without advanced reservation) where the arrival process is nonhomogeneous Poisson (e.g., \cite{eick,eick2}), \cite{massey} for the details on some of the results along these lines). The deterministic counterpart systems with advanced reservation have been considered in the scheduling and parallel computing literature, which is not the main focus of this paper.

The theoretical results in the loss network systems with advanced reservation find interesting applications in a class of revenue management problems. The most relevant prior work in these applications is \cite{levi} which used a simple knapsack-type linear program (LP) to devise a conceptually simple admission control policy called class selection policy (CSP)  for the models without advanced reservation (i.e., all customers wish to start service upon arrival). The optimal solution obtained by solving the LP guides the policy to select the more profitable classes of customers. The LP provides an upper bound on the optimal expected long-run average revenue and can be used to analyze the performance of CSP. The analysis is based on the fact that the CSP induces a stochastic process that can be reduced to a classical loss network model without advanced reservation. They developed explicit expressions for the resulting blocking probabilities induced by the CSP, and then showed that the CSP is guaranteed to achieve at least half of the optimal long-run revenue. Also, the CSP was shown to be asymptotically optimal when the capacity goes to infinity. The knapsack-type LP considered by \cite{levi} has been previously discussed by several other researchers (see, e.g., \cite{key}, \cite{hunt}). In fact, a variant of the CSP has been discussed by \cite{key} and \cite{kelly}, who analyzed the randomized thinning policy. Moreover, \cite{key} has shown that the variant of the CSP for the single resource case without advanced reservation is asymptotically optimal in the critically loaded regime. \cite{iye} have also used an identical LP to devise exponential penalty function control policies to approximately maximize the expected reward rate in a loss network. All of these works have considered models without advanced reservation.

\subsection{Structure of this paper}
The remainder of the paper is organized as follows. In Section  \ref{sec_model}, we describe the mathematical model. In Section \ref{sec_analysis}, we present an asymptotic analysis on the blocking probabilities in loss network systems with advanced reservation. In Section \ref{sec_rm}, we focus on applications in revenue management of reusable resources with advanced reservation and draw connections between these applications and our theoretical findings in loss network systems. We propose an improved class selection policy and show that it performs asymptotically near-optimal. We extend our model to a pricing model in Section \ref{sec_pricing}. In Section \ref{sec_NE}, we present the dynamic programming formulation and then demonstrate the empirical effectiveness of our policy. Finally we conclude the paper with some future research directions in Section \ref{sec_con}. The proofs of technical lemmas and propositions are provided in the Appendix.

\section{Loss Network Systems with Advanced Reservation} \label{sec_model}

We first describe a well-structured stochastic process called loss network systems with advanced reservation, i.e., \emph{M/G/C/C} loss systems with advanced reservation. We consider a homogeneous pool of resources of integer capacity $C <\infty$ being used to satisfy the demands of $M$ different classes of customers. The customers of each class $k=1,\ldots,M$, arrive according to an independent Poisson process with a class-dependent rate $\lambda_{k}$. Each class-$k$ customer requests to reserve one unit of the capacity for a specified \emph{service time interval} in the future.

Let $D_{k}$ be the reservation distribution of a class-$k$ customer, and $S_{k}$ be the respective service distribution with mean $\mu_{k}$ (see Figure \ref{model}). In particular, upon an arrival of a class-$k$ customer at some random time $t$, the customer requests to reserve the service time interval $[t + d, t + d + s]$, where $d$ and $s$ are drawn according to $D_{k}$ and $S_{k}$, respectively. Thus, $t+d$ is called the \emph{starting service time}. Note that $D_{k}$ and $S_{k}$ are independent of the arrival process and between customers; however, per customer, $D_{k}$ and $S_{k}$ can be correlated. We assume that $D_{k}$ is discrete with finite support $[0, u_{k}]$ and $S_{k}$ is discrete with finite support $[1,v_{k}]$. 

\begin{figure}[ht]
\centering
\includegraphics[scale=0.7]{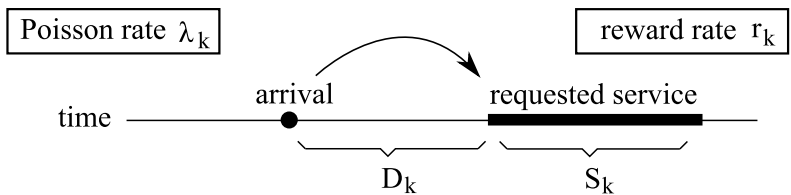}
\caption{Reservation distributions and service distributions}
\label{model}
\end{figure}

The resource unit can be reserved for an arriving customer only if upon arrival there is at least one unit of capacity that is available (i.e., not reserved before time $t$) throughout the entire requested interval $[t + d, t + d + s]$, i.e., this request can only be satisfied if the maximum number of already reserved resource units over $[t + d, t + d + s]$ is strictly smaller than the capacity $C$. Customers whose requests are not reserved upon arrival due to insufficient capacity are \emph{lost} and leave the system. 

Upon reservation, the requested unit will then be fully committed and cannot be canceled for the reserved service interval. In addition, there is a class-specific reward rate of $r_{k}$ collected per unit of service time. The key research question in loss network systems with advanced reservation is to characterize or upper bound the \emph{blocking probabilities}, i.e., the long-run probabilities that a random customer sees a fully utilized system and is therefore blocked for service.

\section{An Asymptotic Analysis of Blocking Probabilities} \label{sec_analysis}

Before delving into details, we first give an overview of our analysis. Analyzing the original capacitated system seems rather difficult. Instead, we consider the counterpart system with infinite capacity (i.e., a \emph{M/G/$\infty$} system with advanced reservation) while keeping all other problem parameters fixed. In this counterpart system, all customers are admitted since there is an infinite number of resources. It is not hard to see that, for each sample path and each time $t$, the admitted customers reserved to get service in the original capacitated system are a subset of those reserved in the infinite capacity counterpart system. Consider now a customer arriving at some random time $t$ in the counterpart system with infinite capacity requesting service interval $[t + d, t + d + s]$. Define the \emph{virtual blocking probability} to be the probability that the maximum reserved capacity over the requested service interval $[t + d, t + d + s]$ just prior to time $t$ is larger than $C$. Since the set of served customers in the infinite capacity system is always a superset of that served in the original capacitated system, it follows that the virtual blocking probability is in fact an upper bound on the blocking probability in the original capacitated system. It makes sense to analyze the asymptotic behavior of the virtual blocking probabilities, which, in turn, provides us asymptotic upper bounds on the blocking probabilities in the original capacitated system. Let us start with the simplest non-trivial case (which gives us some insights into how to analyze such complex models), and gradually develop our main result for the general case.

\subsection{The Simplest Non-Trivial Case: Two-Point Distribution}
We will start the asymptotic analysis with the simplest non-trivial case, and then extend it gradually to the more general case. Suppose that $S$ takes only one value $s=1$ deterministically. Then the traffic intensity $\rho = \lambda \mu = \lambda$. In addition, assume that $D$ follows a two-point distribution,
\begin{equation*}
D =
\begin{cases}
0 & \text{w.p. } \gamma,\\
1 & \text{w.p. } 1-\gamma,
\end{cases}
\end{equation*}
i.e., $f_{D}(0) = \gamma$ and $f_{D}(1) = 1-\gamma$. That is, an arriving customer either wants to start the service immediately or in $1$ unit of time. Consider the counterpart system with an infinite number of servers in steady-state (note that the steady-state exists due to the induced semi-Markov process). Upon a customer arrival to the system at some time $t$, all the starting service times of the customers who had arrived prior to $t$ are already known. For ease of exposition, we call these starting service times \emph{pre-arrivals}. Similarly, we call all the starting service times of the customers, who will arrive after $t$ \emph{post-arrivals}. Note that the pre-arrivals and post-arrivals are always defined with respect to the current time $t$. It is important to observe that the virtual blocking probability at time $t$ (as well as the blocking probability in the original capacitated system) is independent of \emph{post-arrivals}. Without loss of generality, we assume that $t=0$ and the system reaches equilibrium. 

Lemma \ref{firstl} below characterizes the pre-arrival processes (i.e., the booking profile) observed by a customer arriving at time $0$ in steady-state. Let $\left\lceil r \right\rceil$ is the smallest integer not less than $r$. 
\begin{lemma}
\label{firstl}
Consider the counterpart system with an infinite number of servers, then a customer arriving at the system at time $0$ in steady-state, observes that the pre-arrivals follow a non-homogeneous Poisson process with piecewise rate $\eta(r)$ at time $r$
\begin{equation*}
\label{sst}
\eta(r) =
\begin{cases}
\lambda, & \text{if} \qquad r \le 0,\\
(1-\gamma)\lambda, & \text{if} \qquad 0 < r \le 1,\\
0, & \text{if} \qquad r > 1.
\end{cases}
\end{equation*}
\end{lemma}

The proof of Lemma \ref{firstl} is simple by using Poisson splitting arguments. In order to figure out how likely this customer (arriving in equilibrium) gets blocked, it is important to know the entire booking profile (consisting of committed services not yet started) at the moment of his or her arrival. Lemma \ref{firstl} gives a compact description of this pre-arrival process as seen from $t=0$.

Let $N_{d}(r)$ where $r \in [0,1]$ be the Poisson counting process of the number of pre-arrivals over the interval $[d-1, d]$ \emph{as seen from time 0}. The corresponding rates of this Poisson counting process are given by Lemma \ref{firstl}. Next, we introduce the notion of \emph{mirror image} of a Poisson counting process. The mirror image of a Poisson counting process $N_{d}(r)$, denoted by $\tilde{N}_{d}(r)$, is a backward counting process of $N_{d}(r)$. More formally, $\tilde{N}_{d}(r) = N_{d}(1) - N_{d}(1-r)$ for each $r \in [0,1]$. It is evident that $\tilde{N}_{d}$ is also a Poisson process with the same rate as $N_{d}$. We will use  $\tilde{N}_{d}(r)$ to model the departure process over the interval $[d, d+1]$ in reverse time.

Now let $B$ be the event that a customer arriving at time $0$ in steady-state is virtually blocked. The conditional long-run virtual blocking probability $P_{d} \triangleq \mathbb{P}(B \mid D=d)$, for each $d=0,1$. Lemma \ref{sl1} below characterizes $P_{0}$ and $P_{1}$ based on the counting processes introduced above. Figure \ref{adv3} gives a schematic representation of the two processes as observed a random customer arriving at $t=0$. For clarity, we also use $N_{d}( \cdot ;\lambda)$ to denote a Poisson counting process with a given rate $\lambda$. 

\begin{figure}[ht]
\centering
\includegraphics[scale=0.5]{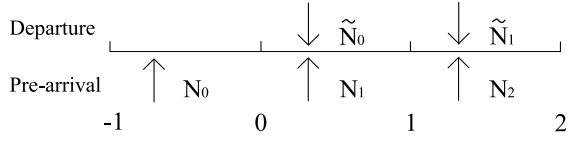}
\caption{One-class departure and pre-arrival processes}
\label{adv3}
\end{figure}

\begin{lemma}
\label{sl1}
Consider the counterpart system with an infinite number of servers, if a customer arrives at time $0$ in steady-state and requests service $S=1$ deterministically to commence in $D$ units of time ($D=0$ or $1$ with probabilities $\gamma$ and $1-\gamma$, respectively), the conditional virtual blocking probabilities are given by
\begin{eqnarray*}
P_{0} &\triangleq& \mathbb{P}(B \mid D=0) = \mathbb{P} \left(\max_{r \in [0,1]}
\left\{ \tilde{N}_{0}(1-r; \lambda) + N_{1}(r;(1-\gamma)\lambda) \right\} \ge C \right),\\
P_{1} &\triangleq& \mathbb{P}(B \mid D=1) = \mathbb{P} \left( \tilde{N}_{1}(1;(1-\gamma)\lambda)   \ge C \right),
\end{eqnarray*}
where $N_{d}$ ($d=0,1,2$) is a Poisson counting process with rate  $\lambda_{d}$ and $\tilde{N}_{d}$ ($d=0,1$) is the mirror image of $N_{d}$ with the same rate $\lambda_{d}$. 
Moreover, $\lambda_{0}= \lambda$, $\lambda_{1}=(1-\gamma) \lambda$ and $\lambda_{2} = 0$. 
\end{lemma}

Lemma \ref{sl1} essentially tells us that the virtual blocking probability can be expressed in terms of the maximum of the sum of these two Poisson counting processes running towards each other (see Figure \ref{adv2}), one of which representing the pre-arrival process (the committed services not yet started) and the other one representing the departure process. The random process (inside the max operator) is hard to analyze, which is very different than merging two Poisson counting processes running in the same direction.

\begin{figure}[ht]
\centering
\includegraphics[scale=0.63]{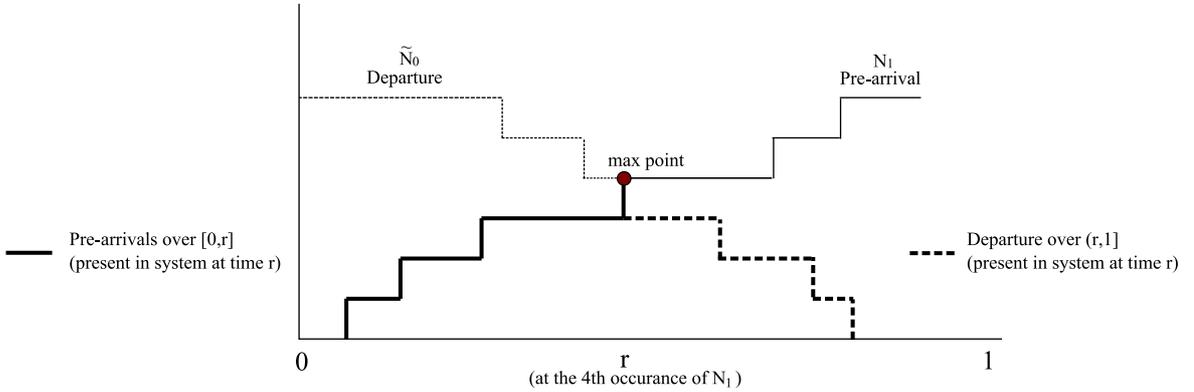}
\caption{Two Poisson counting processes running towards each other}
\label{adv2}
\end{figure}

To analyze these blocking probabilities, we first assume that the probability that an arriving customer seeks to start service immediately is positive (i.e., $\gamma > 0$), and then drop this assumption later. Now assuming that $\gamma >0$, we shall show that under the conventional heavy-traffic regime where both the arrival rate and the capacity scale together to infinity, i.e., $C  = \lambda \rightarrow \infty$, the conditional virtual blocking probabilities $P_{0}$ and $P_{1}$ have the following asymptotic limits, namely,
$\lim_{\lambda \rightarrow \infty} P_{0} = 1/2$ and $\lim_{\lambda \rightarrow \infty} P_{1} = 0$. In fact, we shall prove a more general statement that will be useful in our analysis of the general case. 

\begin{proposition}
\label{primary}
Let $N_{0}$, $N_{1}$ and $N_{2}$ be Poisson counting processes (mutually independent) with rates $\lambda$, $\theta_{1} \lambda$ and $\theta_{2} \lambda$, respectively, where $1>\theta_{1}\ge \theta_{2} \ge 0$ are fixed constants. Let $\tilde{N}_{0}$ and 
$\tilde{N}_{1}$ be the mirror images of $N_{0}$ and $N_{1}$, respectively. Define two random variables $X$ and $Y$ as follows,
\begin{eqnarray}
\label{def-x-y}
X \triangleq  \max_{r \in [0,1]} \left\{ \tilde{N}_{0}(1-r; \lambda) + N_{1}(r;\theta_{1} \lambda) \right\}, \qquad
Y \triangleq  \max_{r \in [0,1]} \left\{ \tilde{N}_{1}(1-r; \theta_{1} \lambda) + N_{2}(r;\theta_{2} \lambda) \right\}. 
\end{eqnarray}
Then, in the high-volume regime where $C  = \lambda \rightarrow \infty$,
$$
\lim_{\lambda \rightarrow \infty} \mathbb{P}(X \ge C) = \frac{1}{2}, \qquad \lim_{\lambda \rightarrow \infty} \mathbb{P}(Y \ge C) = 0.
$$
\end{proposition}

Note that $P_{0}$ and $P_{1}$ (in Lemma \ref{sl1}) can be obtained by simply setting $\theta_{1} = (1-\gamma)$ and $\theta_{2} = 0$ in $X$ and $Y$ above. We observe that asymptotically only customers with zero reservation time are likely to be blocked. This stems from the fact that for a given time slot and unit capacity, these customers are the last to arrive. Since we have a system where the accepted demand is close to the capacity, these are the customers that will most likely be blocked.

To prove Proposition \ref{primary}, we provide an alternative characterization of $X$ and $Y$ above based on a downward-drifting asymmetric random walk process that takes a down-step, for each departure, and an up-step, for each pre-arrival. We would like to show that asymptotically the maximum level of the random walk stays relatively close to its starting position by showing that the rate of the random walk going up is sublinear in $\sqrt{\lambda}$. 

Consider the merged process induced on $[0,1]$ by the two Poisson counting processes $\tilde{N}_{0}$ and $N_{1}$. Let  $\mathcal{N} = \tilde{N}_{0}(1;\lambda) + N_{1}(1;\theta \lambda)$ denote the total number of occurrences (pre-arrivals and departures) over $[0,1]$ of the two independent Poisson counting processes of $\tilde{N}_{0}$ and $N_{1}$. Note that since $\tilde{N}_{0}$ and $N_{1}$ are independent of each other, $\mathcal{N}$ is a Poisson random variable with rate $(1+\theta)\lambda$. Conditioning on $\mathcal{N}=n$, the induced merged process has $n$ points uniformly distributed over the interval $[0,1]$. By the splitting argument applied to the merged process, each of these $n$ points has independent probability $p = \frac{\theta \lambda}{(1+\theta) \lambda} = \frac{\theta}{1+\theta} < \frac{1}{2}$
to be from the process $N_{1}$ and probability $q = 1-p $ from the process $\tilde{N}_{0}$. If we associate $+1$ with each point from $N_{1}$, and $-1$ with each point from $\tilde{N}_{0}$, then each configuration of these $n$ points induces a downward-drifting asymmetric random walk of length $n$. The random walk starts at the origin $0$, with up probability $p$ and down probability $q$. 

\begin{lemma}
\label{decomposition}
Let $\mathcal{R}_{n}$ denote the corresponding random walk of length $n$ described above, and $M_{n}$ denote the maximum level attained by $\mathcal{R}_{n}$, and $G_{n}$ denote the overall number of down-steps taken by $\mathcal{R}_{n}$. Also let $X_{n} \triangleq (X \mid \mathcal{N} = n)$ with $X$ defined in (\ref{def-x-y}). Then, $X_{n} = G_{n} + M_{n}$ almost surely.
\end{lemma}

However, it should be noted that $M_{n}$ and $G_{n}$ are correlated. To address the correlation between $M_{n}$ and $G_{n}$, we will replace $M_{n}$ by $M_{\infty}$. However, first we would like to obtain an expression for the hitting probability of a downward-drifting asymmetric random walk. This is done in Lemma \ref{hitting} given below. (\cite{lawler} provided a proof in Chapter 2, Section 2.2; for completeness, we present a shorter proof.)

\begin{lemma}
\label{hitting}
Consider a random walk defined by a sequence of independent random variables $E_{i}=1$ with probability $p$ and $-1$ with probability $q=1-p$. Let $S_{n} = \sum_{i=1}^{n} E_{i}$. Define $M_{\infty} \in [0,\infty) \bigcup \{\infty\}$ to be maximum level attained by the random walk (i.e., $M_{\infty}=\max_{n}S_{n}$). Given that $0\le p<q \le 1$ (downward drifting), then the probability that the random walk ever hits above level $b$ is
$\mathbb{P} \left(M_{\infty} \ge b \right) = \left(p/q \right)^{b}.$
\end{lemma}

We are now ready to prove Proposition \ref{primary}.
\begin{proof}[Proof of Proposition \ref{primary}.]
First we shall prove that $\lim_{\lambda \rightarrow \infty} \mathbb{P}(X \ge C) = 1/2$.
Let $M_{\infty}$ be the maximum level attained by the infinite-step random walk defined above. Since the random walk has a negative drift, it follows from Lemma \ref{hitting} above that $\mathbb{P}(M_{\infty} \ge -\log \lambda/\log \theta) \le 1/\lambda.$ (Note that
$\theta <1$, so $-\log \lambda / \log \theta > 0$.) Now, we have
\begin{eqnarray}
\label{imp1}
\mathbb{P}(X_{n} \ge C) 
&=& \mathbb{P}\left(G_{n} + M_{n} \ge C \right) \\
&=& \mathbb{P}\left(G_{n} + M_{n} \ge C \bigcap M_{n} \ge -\frac{\log \lambda}{\log \theta}\right)  \nonumber 
+ \mathbb{P}\left(G_{n} + M_{n} \ge C \bigcap M_{n} < -\frac{\log \lambda}{\log \theta} \right)  \nonumber \\ 
&\le& \mathbb{P}\left(M_{n} \ge -\frac{\log \lambda}{\log \theta}\right) + \mathbb{P}\left(G_{n} \ge C + \frac{\log \lambda}{\log \theta}\right) 
\nonumber \\
&\le& \mathbb{P}\left(M_{\infty} \ge -\frac{\log \lambda}{\log \theta}\right) + \mathbb{P}\left(G_{n} \ge C + \frac{\log \lambda}{\log \theta}\right) \nonumber \\
&\le& \frac{1}{\lambda} + \mathbb{P}\left(G_{n} \ge C + \frac{\log \lambda}{\log \theta}\right). \nonumber 
\end{eqnarray}
The first equality follows from Lemma \ref{decomposition}. The first inequality follows from the fact that $M_{\infty} \ge M_{n}$ almost surely. The second inequality follows from Lemma \ref{hitting} above. Since $G_{n}$ is distributed as $(\tilde{N}_{0}(1; \lambda) \mid \mathcal{N}=n)$, we get from (\ref{imp1}) that,
\begin{eqnarray*}
\mathbb{P}(X \ge C) &=& \sum_{n=1}^{\infty} \mathbb{P}(X_{n} \ge C) \mathbb{P}(\mathcal{N}=n)  
\le \frac{1}{\lambda} + \sum_{n=1}^{\infty}  \mathbb{P}\left(G_{n} \ge C + \frac{\log \lambda}{\log \theta}\right) \mathbb{P}(\mathcal{N}=n)\nonumber \\
&=&\frac{1}{\lambda} + \mathbb{P}\left(\tilde{N}_{0}(1; \lambda) \ge C + \frac{\log \lambda}{\log \theta}\right) 
= \frac{1}{\lambda} + \mathbb{P}\left(\mathrm{Poisson}(\lambda)  \ge C + \frac{\log \lambda}{\log \theta}\right) . \nonumber 
\end{eqnarray*}
By virtue of Central Limit Theorem, we have 
\begin{eqnarray*}
\lim_{\lambda \rightarrow \infty} \mathbb{P}(X \ge C) &\le& 
\lim_{\lambda \rightarrow \infty} \mathbb{P}\left(\mathrm{Poisson}(\lambda)  \ge C - \frac{\log \lambda}{\log \theta^{-1}}\right) \\
&=& \lim_{\lambda \rightarrow \infty} \mathbb{P}\left(\mathrm{Poisson}(\lambda)  \ge \lambda - o(\sqrt{\lambda}) \right)  = \frac{1}{2}.
\end{eqnarray*}
On the other hand, from the definition of $X$, it is clear that
$\lim_{\lambda \rightarrow \infty} \mathbb{P}(X \ge C) \ge  \lim_{\lambda \rightarrow \infty} \mathbb{P}(N_{0}(1;\lambda) \ge C) = 
\lim_{\lambda \rightarrow \infty} \mathbb{P}(\mathrm{Poisson}(\lambda) \ge \lambda ) = 1/2$.

Now, we are ready to prove the second part of Proposition \ref{primary}, i.e., $\lim_{\lambda \rightarrow \infty} \mathbb{P}(Y \ge C) = 0$.
Since $\theta_{1}\in [0,1)$, we can always find a $\bar{\theta}_{1}$ such that $\theta_{1} < \bar{\theta}_{1} < 1$. Then $\bar{\theta}_{1} > \theta_{1} \ge \theta_{2}$, and define
$\bar{Y} = \max_{r \in [0,1]} \left\{ \bar{N}_{1}(1-r; \bar{\theta}_{1} \lambda) + 
N_{2}(r; \theta_{2} \lambda) \right\}.$
It is easy to see that $\bar{Y}$ stochastically dominates $Y$. Therefore, without loss of generality, we simply drop the bar of $\bar{\theta}_{1}$ and $\bar{Y}$, and assume that $\theta_{2} / \theta_{1} = \theta <1$. Following the same argument above, we have
$\mathbb{P}(Y \ge C) \le \frac{1}{\lambda} + \mathbb{P}\left(\mathrm{Poisson}(\theta \lambda) \ge \lambda 
+ \frac{\log \lambda}{\log \theta}\right),$ and again by virtue of Central Limit Theorem, we have
\begin{eqnarray*}
\lim_{\lambda \rightarrow \infty} \mathbb{P}(Y \ge C) &\le& 
\lim_{\lambda \rightarrow \infty} \mathbb{P}\left(\mathrm{Poisson}(\theta \lambda)  \ge C - \frac{\log \lambda}{\log \theta^{-1}}\right) \\
&=& \lim_{\lambda \rightarrow \infty} \mathbb{P}\left(\mathrm{Poisson}(\theta  \lambda)  \ge \lambda  - o(\sqrt{\lambda}) \right)  =0.
\end{eqnarray*}
This completes the proof. 
\end{proof}

We have shown that if $0<\gamma \le 1$, $\lim_{\lambda \rightarrow \infty} P_{0} \le 1/2$, and $\lim_{\lambda \rightarrow \infty} P_{1} =0$. Now if $\gamma = 0$, it follows that no arriving customer will start service immediately. Thus, we have $\lim_{\lambda \rightarrow \infty} P_{1} = 1/2$.

\subsection{The General Case}
Next we extend the simple model to allow for an arbitrary finite discrete reservation distribution $D$ with marginal probability mass function $f_{D}(d)$. We still assume that the service distribution remains fixed at $S=1$, deterministically. Now let  $f_{D}(d) = \gamma_{d}$ for $d \in [0,u]$, $0\le \gamma_{d}\le 1$ and $\sum_{d=1}^{u} \gamma_{d}=1$. Lemma \ref{ext1} below is a generalization of Lemma \ref{sst}. 
\begin{lemma}
\label{ext1}
Consider the counterpart system with an infinite number of servers, a customer arriving at the system at time $0$ in steady-state, observes that the pre-arrivals follow a non-homogeneous Poisson input process with piecewise rate $\eta(r)$ at time $r$
\begin{equation*}
\label{sstg}
\eta(r) =
\begin{cases}
\lambda, & \text{if} \qquad r \le 0,\\
\lambda \left(1- \sum_{d= 0}^{\left\lceil r \right \rceil  -1} \gamma_{d}\right), & \text{if} \qquad r>0.
\end{cases}
\end{equation*}
\end{lemma}

\begin{figure}[ht]
\centering
\includegraphics[scale=0.5]{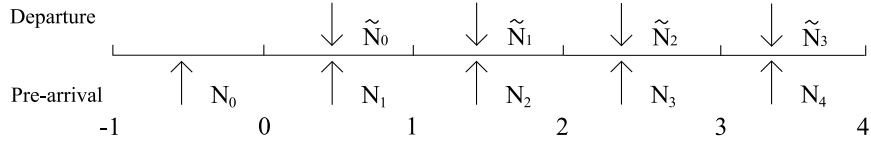}
\caption{One-class departure and pre-arrival processes with general reservation distribution}
\label{adv4}
\end{figure}

Define $N_{d}$ (for $d\in [0,u]$) to be the process of pre-arrivals prior to $t$ over  $(d-1, d]$. This process induces a departure process over the interval $(d, d+1]$, and let $\tilde{N}_{d}$ denote its mirror image. Figure \ref{adv4} shows the pre-arrival and departure processes with general reservation distribution. The conditional virtual blocking probabilities are given in Lemma \ref{gl1} below, which is a generalization of Lemma \ref{sl1}. 

\begin{lemma}
\label{gl1}
Consider the counterpart system with an infinite number of servers, if a customer comes at time $0$ in steady-state and requests service ($S=1$) deterministically to commence in $D$ units of time ($D\in [0,u]$), the conditional virtual blocking probability is given by
\begin{eqnarray*}
P_{d} &\triangleq& \mathbb{P}(B \mid D=d) \triangleq \mathbb{P} \left( \max_{r \in [0,1]}
\left\{ \tilde{N}_{d}(1-r;\lambda_{d}) + N_{d+1}(r;\lambda_{d+1}) \right\} \ge C \right), \qquad \forall d \in [0,u],
\end{eqnarray*}
where $N_{d}$ is a Poisson counting process with rate $\lambda_{d}=\lambda\left(1- \sum_{i=0}^{d-1}\gamma_{i}\right)$, and $\tilde{N}_{d}$ is a mirror image of $N_{d}$ with the same rate $\lambda_{d}$.
\end{lemma}

Proposition \ref{gl2} is a generalization of Proposition \ref{primary} with general reservation distribution. The traffic intensity $\rho = \lambda$ since the service distribution $S=1$ deterministically. 
\begin{proposition}
\label{gl2}
Let the service distribution $S=1$ deterministically, and the reservation distribution $D$ be discrete with marginal probability mass function $f_{D}(d)=\gamma_{d}$ and bounded support $[0,u]$. Let the index $i = \min\{d: \gamma_{d} >0\}$. 
Then, in the high-volume regime where $C  = \rho \rightarrow \infty$, the conditional long-run virtual blocking probabilities  
\begin{eqnarray*}
\lim_{\lambda \rightarrow \infty} P_{i} = \frac{1}{2}; \qquad
\lim_{\lambda \rightarrow \infty} P_{d} = 0, \text{ for } d \in [i+1,u].
\end{eqnarray*}
\end{proposition}

Next we extend the model further to allow for an arbitrary finite discrete service distribution. The total arrival rate is $\lambda$, and the reservation distribution $D$ is defined on $[0,u]$ defined as in Section 3.2. Now assume that the service time $S$ is a general finite discrete distribution on $[1,v]$. More specifically, let $f_{S}(\cdot)$ be the marginal probability mass function with $f_{S}(s)= \mathbb{P}(S=s) =\kappa_{s}$, where $\sum_{s=1}^{v}\kappa_{s} =1$ and $0\le \kappa_{s} \le 1$, for each $s\in [1,v]$. 

We partition the arriving customers according to their requested service time, i.e., the customers are partitioned into $v$ disjoint sets numbered $1,\ldots,v$ according to their requested service time. For each $s\in [1,v]$, the arrival process of customers in set $s$ follows a thinned Poisson process with rate $\kappa_{s}\lambda$. Moreover, these processes are independent of each other. Now, for each set $s \in [1,v]$, let the conditional reservation probability mass function be $\mathbb{P}(D=d \mid S=s) = \gamma^{s}_{d}$ for $d \in [0,u]$. Note that $\sum_{d=0}^{u} \gamma^{s}_{d} = 1$, for each $s \in [1,v]$.

Consider the counterpart system with an infinite number of servers, if a customer of set $s$ ($s\in [1,v]$) arrives at time $0$ in steady-state and requests $s$ units of service time to commence after $d$ units of time ($d\in [0,u]$), the conditional virtual blocking probability is defined as $P_{d}^{s} \triangleq \mathbb{P}(B \mid D=d, S=s)$. In addition, the traffic intensity is $\rho = \sum_{s=1}^{v} s \kappa_{s} \lambda = \mu \lambda$, where $\mu = \sum_{s=1}^{v}  s \kappa_{s} $ is the mean service time. 

Let $N_{d}^{s}$ (for $s \in [1,v]$ and $d \in [0,u]$) denote the pre-arrival process of set-$s$ customers over (i.e., customers requesting $s$ units of service time) the interval $(d-s,d-s+1]$. This induces a departure process over the interval $(d,d+1]$, and let $\tilde{N}_{d}^{s}$ denote its mirror image. The rate of $N^{s}_{d}$ ( and $\tilde{N}^{s}_{d}$) is given in Lemma \ref{rs} below. 

\begin{lemma}
\label{rs}
Let $N^{s}_{d}$ and $\tilde{N}^{s}_{d}$ be defined as above. Then, for each $s \in [1,v]$ and each $d \in [0,u]$, $N^{s}_{d}$ and $\tilde{N}^{s}_{d}$ are Poisson processes with the same rate
\begin{equation}
\label{rates}
\lambda^{s}_{d} = \lambda_{0}^{s} \left(1 - \sum_{i=0}^{d-s} \gamma_{i}^{s}\right) = \kappa_{s}\lambda \left(1 - \sum_{i=0}^{d-s} \gamma_{i}^{s}\right).
\end{equation}
Moreover, $N_{d}^{s}$ is independent of $N_{d'}^{s'}$ for $d \ne d'$ or $s \ne s'$.
\end{lemma}

\begin{figure}[ht]
\centering
\includegraphics[scale=0.5]{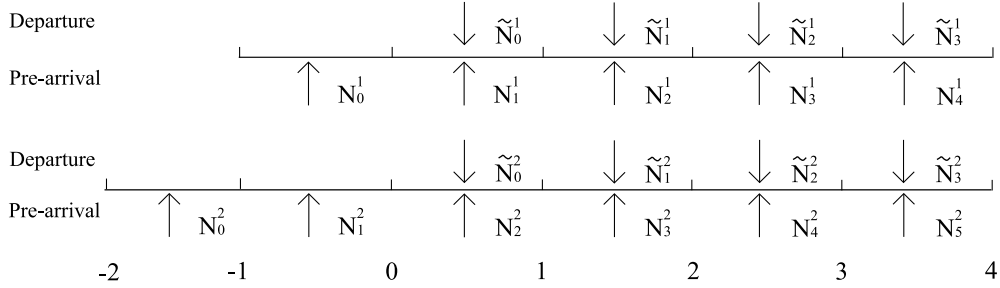}
\caption{Two-service-set departure and pre-arrival processes}
\label{adv1}
\end{figure}

First assume that $\exists s \in [1,v]$ such that $\gamma_{0}^{s}>0$, i.e., the probability of an arriving customer requesting to start the service immediately upon arrival is strictly positive. This assumption can be dropped later. Let $A_{d}$ be the maximum number of customers in the system over the interval $(d,d+1]$ for $d \in [0,u]$. In fact, one can derive an exact mathematical expression of each $A_{d}$ for $d \in [0,u]$,
\begin{equation}
\label{illustr}
A_{d} = \sum_{s=2}^{v} \sum_{i=d+1}^{d+s-1} N_{i}^{s}(1;\lambda_{i}^{s}) + \max_{r \in [0,1]} 
\left\{ \sum_{s=1}^{v} \tilde{N}_{d}^{s}(1-r;\lambda_{d}^{s}) + \sum_{s=1}^{v}  N_{d+s}^{s}(r;\lambda_{d+s}^{s}) \right\}.
\end{equation}
For $r \in [0,1]$, the term $\sum_{s=1}^{v} \tilde{N}_{d}^{s}(1-r;\lambda_{d}^{s})$ captures all the departures over $(d+r,d+1]$, the term $\sum_{s=1}^{v}  N_{d+s}^{s}(r;\lambda_{d+s}^{s})$ captures all the pre-arrivals over $(d,d+r]$, and the term $\sum_{s=2}^{v} \sum_{i=d+1}^{d+s-1} N_{i}^{s}(1;\lambda_{i}^{s})$ captures all the customers being served over $(d,d+1]$. The sum captures exactly all the customers being served at time $d+r$. It is important to note that since $\tilde{N}_{i}^{s}$ and $N_{i}^{s}$ do not simultaneously appear in $A_{d}$, for each $i \in [0,u]$ and $s \in [1,v]$, all the Poisson counting processes in the expression of $A_{d}$ are independent of each other (see Lemma \ref{rs}).

We shall further explain (\ref{illustr}) by providing the following example when $v=2$ (refer to Figure \ref{adv1}),
\begin{eqnarray*}
A_{0} &=& N_{1}^{2}(1;\lambda_{1}^{2}) + \max_{r\in [0,1]} \left\{ \tilde{N}_{0}^{1}(1-r;\lambda_{0}^{1}) + \tilde{N}_{0}^{2}(1-r;\lambda_{1}^{2}) + N_{1}^{1}(r;\lambda_{1}^{1}) + N_{2}^{2}(r;\lambda_{2}^{2}) \right\}, \\
A_{1} &=& N_{2}^{2}(1;\lambda_{2}^{2}) + \max_{r\in [0,1]} \left\{ \tilde{N}_{1}^{1}(1-r;\lambda_{1}^{1}) + \tilde{N}_{1}^{2}(1-r;\lambda_{1}^{2}) + N_{2}^{1}(r;\lambda_{2}^{1}) + N_{3}^{2}(r;\lambda_{3}^{2}) \right\}, \text{ and so on.}
\end{eqnarray*}
More specifically, $A_{0}$ represents the maximum customers in the system over the interval $(0,1]$. At time $r \in(0,1]$, the number of departures over $(r,1]$ is equal to $\tilde{N}_{0}^{1}(1-r;\lambda_{0}^{1}) + \tilde{N}_{0}^{2}(1-r;\lambda_{0}^{2})$, capturing customers in both sets starting before $0$ and still in the system at time $r$. (Note that the service time is at least $1$.) In addition, the number of pre-arrivals over $(0,r]$ is equal to $N_{1}^{1}(r;\lambda_{1}^{1}) + N_{2}^{2}(r;\lambda_{2}^{2})$, capturing pre-arrivals of customers with service time $1$ and $2$, respectively, starting service over $(0,r]$. Finally, $N_{1}^{2}$ captures set-$2$ customers with service time $2$ who started service within $(-1,0]$. These customers will continue service over the entire interval $(0,1]$. Therefore $N_{1}^{2}(1;\lambda_{1}^{2})$ appears in the expression $A_{0}$ outside the max. The same reasoning applies to $A_{i}$ for each $i\in [1,u]$. 

Now for each  $d \in [0,u]$ and $s \in [1,v]$, we have $P_{d}^{s} = \mathbb{P}(\max (A_{d}, \ldots, A_{d+s-1})\ge C)$. It should be noted that $A_{d}$ and $A_{d'}$ can be correlated. To analyze the limiting behavior of $P_{i}^{j}$, we first analyze the limits of $\mathbb{P}(A_{d}\ge C)$, for each $d \in [0,u]$.

\begin{lemma}
\label{ac}
Assume that there exists $s \in [1,v]$ such that $\gamma_{0}^{s} > 0$. Let $A_{d}$ be defined as in (\ref{illustr}). The traffic intensity is $\rho = \sum_{s=1}^{v}  s \kappa_{s} \lambda$. Then,  we have
\begin{eqnarray}
\lim_{\lambda \rightarrow \infty}  \mathbb{P}(A_{0}\ge C) = \frac{1}{2}; \qquad \lim_{\lambda \rightarrow \infty} \mathbb{P}(A_{d}\ge C) =  0,  
\text{ for } d \in [1,u]. 
\end{eqnarray}
\end{lemma}

\begin{proposition}
\label{sec}
Let the service distribution $S$ be discrete with marginal probability mass function $f_{S}(s)=\kappa_{s}$ and bounded support $[1,v]$, and the reservation distribution $D$ be discrete with marginal probability mass function $f_{D}(d)=\gamma_{d}$ and bounded support $[0,u]$. The traffic intensity is given by $\rho = \sum_{s=1}^{v} s \kappa_{s} \lambda = \sum_{s=1}^{v}  s \lambda^{s}_{0}$. Let the index
$i = \min \left\{d: \gamma_{d}^{s} > 0 \text{ for some } s \in [1,v] \right\}.$ Then in the high-volume regime where $C  = \rho \rightarrow \infty$, for each $s \in [1,v]$, the conditional long-run virtual blocking probabilities  
\begin{eqnarray}
\lim_{\lambda \rightarrow \infty} P_{i}^{s} = \frac{1}{2}; \qquad
\lim_{\lambda \rightarrow \infty} P_{d}^{s} =  0, \text{ for } d\in [i+1,u].
\end{eqnarray}
\end{proposition}

\subsection{Discussions of Our Results}
For each $k=1,\ldots, M$, let $Q_{dsk}$ be the stationary blocking probability of class-$k$ customers with reservation time $d$ and service time $s$, i.e., the stationary probability that a customer of such type arrives at a random time to the system and is rejected due to insufficient capacity at some point within the requested service interval. Note that the blocking probability $Q_{dsk}$ is well-defined since the corresponding underlying stochastic process is Ergodic (see the technical proof in the Appendix). Let $\lambda_{dsk} \triangleq \lambda_{k} \mathbb{P} (D_{k} = d , S_{k} =s)$ be the arrival rate of class-$k$ customers with reservation $d$ and service $s$. Then, the expected long-run average reward rate collected can be expressed as $\sum_{k=1}^{M}\sum_{d,s} r_{k}\lambda_{dsk}j(1-Q_{dsk})$.

If we merge the $M$-class arrival processes, the merged arrival process has an aggregate rate $\lambda = \sum_{k=1}^{M}\lambda_{k}$, and a customer upon arrival has probability of $\lambda_{k}/\lambda$ to be a class-$k$ customer. Define $v = \max_{k} v_{k}$ and $u = \max_{k} u_{k}$. Let $S$ (discrete with finite support $[1,v]$ and mean $\mu$) and $D$ (discrete with finite support $[0,u]$) be the \emph{merged} service and reservation distributions. The joint probability mass function of $S$ and $D$ is $f_{D,S}(d,s) \triangleq \mathbb{P}(D=d,S=s) = \sum_{k=1}^{M} \frac{\lambda_{k}}{\lambda} \cdot \mathbb{P}(D_{k}=d,S_{k}=s)$, for $d \in [0,u]$ and $s \in [1,v]$. Similarly, the marginal probability mass functions of $S$ and $D$ are $f_{S}(s) \triangleq \mathbb{P}(S=s) = \sum_{k=1}^{M} \frac{\lambda_{k}}{\lambda} \cdot \mathbb{P}(S_{k}=s)$ and $f_{D}(d) \triangleq \mathbb{P}(D=d) = \sum_{k=1}^{M} \frac{\lambda_{k}}{\lambda} \cdot \mathbb{P}(D_{k}=d)$, respectively, for $s \in [0,u]$ and $d \in [1,v]$. It is sufficient for the analysis to use only the marginal probability mass functions. This allows for arbitrary correlation between the reservation and service distributions of a given customer.

\begin{proof}[Proof of Theorem \ref{mainresult}]
Theorem \ref{mainresult}(a) follows directly from Proposition \ref{sec}. Theorem \ref{mainresult}(b) follows the identical arguments (by changing $\rho = (1-\epsilon)C$ everywhere) from Propositions \ref{primary}, \ref{gl2} and \ref{sec}.
\end{proof}

Theorem \ref{mainresult} provides the first asymptotic analysis on the blocking probabilities in loss network systems with discrete advanced reservation and service distributions. The performance analysis (on blocking probabilities) is completely different from the one used in \cite{levi} for models without advanced reservation. In the following, we also discuss the key distinction of our results from two related models.

\paragraph{Key Distinction from the Models without Advanced Reservation.}
In contrast to models without advanced reservation (e.g., \cite{levi}), the major challenges in analyzing the blocking probabilities in loss network systems with advanced reservation is that the blocking event depends on the maximum reserved capacity over a particular requested service interval. This requires the characterization of the booking profile (i.e., the pre-reserved arrival and departure processes) over the service interval. As a simple example with capacity $C = 2$ shown in Figure \ref{challenge}, in the models without advanced reservation, it suffices to check the instantaneous load of the system upon arrival of a customer. However, in the models with advanced reservation, we cannot guarantee one's request by merely checking the instantaneous load of the system at her starting service time upon her arrival, because her request may be potentially blocked by reserved slots of those customers who booked prior to her but will start services after her. (In this simple example, the system has only 1 customer in service when her service begins; however, during her requested service interval, there is a point in time that the system has 3 customers (who were reserved before her). Thus, she has to be rejected by the system.) This introduces much difficulties in handling this correlation issue between the incoming requests and the booking profiles.

\begin{figure}[ht]
\centering
\includegraphics[scale=0.7]{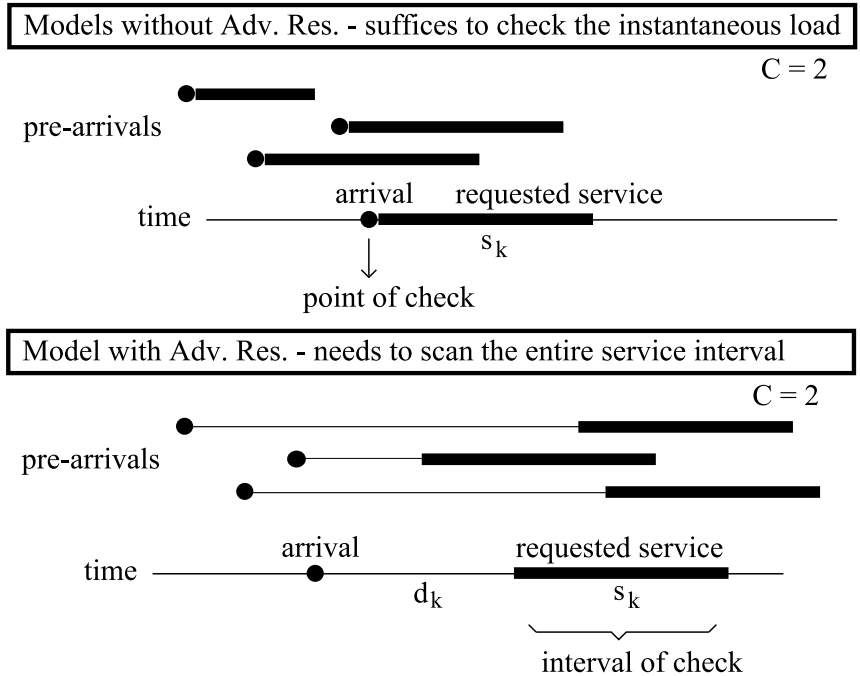}
\caption{Challenges in analyzing the blocking probabilities in loss queues with advanced reservation}
\label{challenge}
\end{figure}

\paragraph{Key Distinction from the Tandem Queues of Two Stations.}
Tandem queueing model is closely related to our model since one may regard the original system with advanced reservation as a tandem queueing model of two stations, where the advance reservation is spent in the first station and the service is spent in the second station. The first station has infinite capacity and the second station has finite capacity, customers first enter the system from the first station, but if the second station is full when customers arrive, they will be rejected or lost. Note that the decision whether a customer is blocked is made only after the service in the first station is over. However, in our model, we have to make the decision as soon as the customer enters the first station. The critical difference is that in the tandem queueing model it suffices to check the instantaneous load of the second station to determine if the customer is blocked, while in our model we have to check the maximum occupancy over the entire requested service interval as seen from the moment the customer arrives. When we relax the finite capacity assumption on the second station, the blocking probability can seemingly be approximated by the probability that the number of customers in the second station is bigger than $C$ (see, e.g., \cite{boxma}, \cite{schmidt}). However, this approximation cannot be used to upper bound the blocking probabilities in our system due to the difference in system dynamics. It may serve as an approximation of the blocking probabilities but we are unsure how good the approximation is, since the stationary distribution can no longer be expressed as a product-form.

\section{Applications in Revenue Management} \label{sec_rm}
\paragraph{Model.}
Next we show how the theoretical results obtained in Section \ref{sec_analysis} can be used to design and analyze provably near-optimal policies for core revenue management models. In particular, we focus our attention on models concerned with revenue management of reusable resources, such as \emph{hotel room management}, \emph{car rental management} and \emph{workforce management}. In all of these settings, customers arrive and seek to book resources some time in advance. Moreover, after a resource unit finishes serving a customer's request, it can be used to serve other customers. The objective is to design an admission control policy that optimizes the expected long-run average revenue rate.

We consider a single pool of resources of integer capacity $C < \infty$ that is used to satisfy the demands of $M$ different classes of customers. The customers of each class $k=1,\ldots,M$, arrive according to an independent Poisson process with respective rate $\lambda_{k}$. Each class-$k$ customer requests to reserve one unit of the capacity for a specified \emph{service time interval} in the future. Similar in Section 2, we let $D_{k}$ be the reservation distribution of a class-$k$ customer, and $S_{k}$ be the respective service distribution with mean $\mu_{k}$. During the time a customer is served, the requested unit cannot be used by any other customer; after the service is over, the unit becomes available again to serve other customers. If the resource is reserved, the customer pays a class-specific rate of $r_{k}$ dollars per unit of service time. The resource can be reserved for an arriving customer only if upon arrival there is at least one unit of capacity that is available (i.e., not reserved) throughout the entire requested interval. Specifically, a customer's request can be satisfied if the maximum number of already reserved resources throughout the requested service interval is smaller than the capacity $C$. However, customers can be rejected even if there is available capacity. Rejecting a customer now possibly enables serving more profitable customers in the future. Customers whose requests are not reserved upon arrival are \emph{lost} and leave the system. The goal is to find a feasible admission policy that maximizes the expected long-run average revenue. Specifically, if $\mathcal{R}_{\pi}(T)$ denotes the revenue achieved by policy $\pi$ over the interval $[0,T]$, then the expected long-run average revenue of $\pi$ is defined as $\mathcal{R}(\pi) \triangleq \lim \inf_{T \rightarrow \infty} (\mathbb{E}[\mathcal{R}_{\pi}(T)]/T)$, where the expectation is taken with respect to the probability measure induced by $\pi$. 

Like many stochastic optimization models, one can formulate this problem using dynamic programming approach. However, even in special cases (e.g., no advanced reservation allowed and with exponentially distributed service times), the resulting dynamic programs seem computationally intractable because the corresponding state space grows very fast. This is known as the \emph{curse of dimensionality}. Thus, finding provably good policies is a challenging task. 

\paragraph{An Improved Class Selection Policy.}
Next we describe a simple linear program (LP) that provides an upper bound on $(1-\epsilon)$ times the achievable expected long-run average revenue for every small positive $\epsilon$. The LP conceptually resembles to the one used by \cite{levi}, \cite{key} and \cite{iye} who study models without advanced reservation. It is also similar in spirit to the one used by \cite{adelman} in the queueing network framework with unit resource requirements again without advanced reservation. We shall show how to use the optimal solution of the LP to construct a simple admission control policy that is called \emph{improved class selection policy} (ICSP). The original class selection policy was first analyzed by \cite{levi} in models without advanced reservation.

At any point of time $t$, the state of the system is specified by the entire booking profile consisting of the class, reservation and service information of each customer in the booking system as well as the customers currently served. Without loss of generality, we restrict our attention to \emph{state-dependent policies}. Note that each state-dependent policy induces a Markov process over the state-space. Moreover, by following similar arguments as in \cite{lura} and \cite{sevastyanov}, one can show that the induced Markov process has a unique stationary distribution which is Ergodic (see the detailed proof in the Appendix).  Since any state-dependent policy induces a Markov process on the state-space of the system that is ergodic, for a given state-dependent policy $\pi$, there exists a long-run stationary probability $\alpha_{dsk}^{\pi}$ for accepting a class-$k$ customer who wishes to start service in $d$ units of time for $s$ units of time, which is equal to the long-run proportion of accepted customers of this type while running the policy $\pi$. In other words, any state-dependent policy $\pi$ is associated with the stationary probabilities $\alpha_{dsk}^{\pi}$ for all possible reservation time $d$, service time $s$ and class $k$. The mean arrival rate of accepted class-$k$ customers with reservation time $d$ and service time $s$ is $\alpha_{dsk}^{\pi} \lambda_{dsk}$. By applying Little's Law and the PASTA property (see \cite{gallager}), the expected number of class-$k$ customers with reservation time $d$ and service time $s$ being served in the system under state-dependent policy $\pi$ is $\alpha_{ijk}^{\pi} \lambda_{ijk} j$. It follows that under policy $\pi$ the expected long-run average number of resource units being used to serve customers can be expressed as $\sum_{k=1}^{M} \sum_{i,j}  \alpha_{ijk}^{\pi} \lambda_{dsk} s$. Fix a small $\epsilon < \min_{k}(\lambda_{k}\mu_{k}C^{-1})$, this gives rise to the \emph{knapsack} LP below:
\begin{eqnarray}
\label{mlp}
\max_{\alpha_{dsk}^{\pi}}  && \sum_{k=1}^{M} \sum_{d,s} r_{k} \alpha_{dsk}^{\pi} \lambda_{dsk} s, \\
\textrm{s.t.} && \sum_{k=1}^{M} \sum_{d,s} \alpha_{dsk}^{\pi} \lambda_{dsk} s \le (1-\epsilon)C,  \nonumber \\
&&  0 \le \alpha_{dsk} \le 1, \text{ } \forall d,s,k.  \nonumber 
\end{eqnarray}
Note that for each feasible state-dependent policy $\pi$, the vector $\alpha^{\pi} = \{\alpha_{dsk}^{\pi}\}$ is a feasible solution for the LP with objective value equal to the expected long-run average revenue of policy $\pi$. 

\begin{lemma}
\label{ub}
Let $\{\alpha^{*}_{dsk}\}$ be the optimal solution of (\ref{mlp}). Then the optimal objective value of (\ref{mlp}) is at least $(1-\epsilon)$ times the optimal expected reveune, i.e.,
$$
\sum_{k=1}^{M} \sum_{d,s} r_{k} \alpha^{*}_{dsk} \lambda_{dsk} s \ge (1-\epsilon)\mathcal{R}(OPT).
$$
\end{lemma}

The LP defined in (\ref{mlp}) can be solved greedily. Without loss of generality, assume that classes are re-numbered such that $r_{1} \ge r_{2} \ge \ldots \ge r_{M}$. Then the optimal solution to (\ref{mlp}) is as follows: there exists a class $M' \le M$ such that $\alpha^{*}_{1}=\ldots = \alpha^{*}_{M'-1} = 1$, $\alpha^{*}_{M'} = 1- \epsilon C(\lambda_{M'}\mu_{M'})^{-1}$, $\alpha^{*}_{M'+1}=\ldots = \alpha^{*}_{M} = 0$.  This solution from solving the LP gives rise to the ICSP: 
\begin{enumerate}[(1)]
\item For each $k=1,\ldots,M'-1$, accept the customer upon arrival if there is sufficient unreserved capacity throughout the requested service interval.
\item If $k=M'$, accept with probability $1- \epsilon C(\lambda_{M'}\mu_{M'})^{-1}$ if there is sufficient unreserved capacity throughout the requested service interval. 
\item For $k=M'+1,\ldots,M$, reject.
\end{enumerate}
The ICSP is conceptually simple in that it splits the classes into profitable and nonprofitable that should be ignored. In fact, we can assume, without loss of generality, that there is no fractional variable in the optimal solution $\alpha^{*}$, i.e., for each $k=1,\ldots,M'$, $\alpha^{*}_{k}=1$. (If $\alpha^{*}_{M'}$ is fractional, we think of class $M'$ as having an arrival rate $\lambda'_{M'} = \alpha^{*}_{M'} \lambda_{M'}$ and then eliminate the fractional variable from $\alpha^{*}$.) 

Since the ICSP accepts the profitable classes $1$ to $M'$ and rejects the nonprofitable classes $M'+1$ to $M$, it induces a well-structured stochastic process called loss network systems with advanced reservation analyzed in Section \ref{sec_analysis}. Each class $k=1,\ldots, M$ induces a Poisson arrival stream with respective rate $\alpha_{k}^{*}\lambda_{k}$, $1\le k \le M$. Thus, for each class $k$ with $\alpha_{k}^{*}=1$, the arrival process is identical to the original process, each class $k$ with $\alpha_{k}^{*}=0$ can be ignored. This gives rise to system dynamics analogous to the loss network model with advanced reservation defined in Section \ref{sec_analysis}. The key aspect of performance analysis of the ICSP boils down to finding an upper bound on the blocking probabilities in these loss network models.

\begin{proof}[Proof of Theorem \ref{mmthm}.]
The expected long-run average revenue of the ICSP is $\sum_{k=1}^{M'}\sum_{d,s} r_{k}\lambda_{dsk}s(1-Q_{dsk})$, where $Q_{dsk}$ is the stationary probability of blocking a class-$k$ customers with reservation time $d$ and service time $s$. However, $\sum_{k=1}^{M'}\sum_{d,s} r_{k}\lambda_{dsk}s$ is the optimal value of the LP defined in (\ref{mlp}), which is an upper bound on $(1-\epsilon)\mathcal{R}(OPT)$ for small positive $\epsilon$ by Lemma \ref{ub}. Thus, a key aspect of the performance analysis of the ICSP is to obtain an upper bound on the probabilities $Q_{dsk}$'s. Specifically, if $1-Q_{dsk} \ge \xi$, for each $d$, $s$, and $k$, then
\begin{equation*}
\mathcal{R}(ICSP) = \sum_{k=1}^{M'}\sum_{d,s} r_{k}\lambda_{dsk}s(1-Q_{dsk}) \ge \sum_{k=1}^{M'}\sum_{d,s} r_{k}\lambda_{dsk}s \xi \ge \xi (1-\epsilon)\mathcal{R}(OPT).
\end{equation*}
By Theorem \ref{mainresult}(b), we know that $\xi \rightarrow 1$ as $\rho \rightarrow \infty$. This completes the proof. 
\end{proof}

\section{Pricing Extensions} \label{sec_pricing}
We present an interesting pricing extension, in which the arrival rates of the different classes of customers are affected by prices. Specifically, consider a two-stage decision. At the first stage, we set the respective prices $r_{1}, \ldots, r_{M}$ for each class. This determines the respective arrival rates $\lambda_{1}(r_{1}),\ldots,\lambda_{M}(r_{M})$. (The rate of class-$k$ customers is affected only by price $r_{k}$.) Then, given the arrival rates, we wish to find the optimal admission policy that maximizes the expected long-run revenue rate. In particular, we assume that $\lambda_{k}(r_{k})$ is nonnegative, differentiable, and decreasing in $r_{k}$ for each $1 \le k \le M$. In addition, we assume that all prices are nonnegative real numbers and that there exists a price $r_{\infty}$ such that, for each $i=k,\ldots, M$, we have $\lambda_{k}(r_{\infty})=0$. (The latter condition is required to guarantee that the problem has an optimal solution.)

For the model with price-driven demand we use the following nonlinear program (NLP1):
\vspace{-1.5mm}
\begin{eqnarray}
\label{nlp1}
\max_{\alpha_{dsk}, r_{k}} && 
\sum_{k=1}^{M} \sum_{d,s} r_{k} \alpha_{dsk} \lambda_{dsk}(r_{k})  s,  \\
\textrm{s.t.} && \sum_{k=1}^{M} \sum_{d,s} \alpha_{dsk} \lambda_{dsk}(r_{k}) s \le ( 1-\epsilon) C, \nonumber \\
&& 0 \le \alpha_{dsk} \le 1, \qquad \forall d,s,k, \nonumber \\
&& 0 \le r_{k} \le 1, \qquad \forall k. \nonumber
\end{eqnarray}
In particular, it can be verified that any optimal solution of (NLP1) has only nonnegative prices. Also, observe that for any fixed prices $r_{1}, \ldots, r_{M}$, the corresponding solution of $\{\alpha_{dsk}\}$ has the same knapsack structure defined in Section \ref{sec_rm} above. Let $(r^{*},\alpha^{*})= \{r_{k},\alpha_{dsk}\}$ be the corresponding optimal solution. Note that if one can solve (NLP1) and obtain the solution $(r^{*},\alpha^{*})$ then one can construct a similar ICSP that will be amenable to the same performance analysis discussed in Section 3 above. However, solving (NLP1) directly may be computationally hard. Next, we show that one can reduce (NLP1) to an equivalent nonlinear program that is more tractable; we denote it by (NLP2). (By equivalent we mean that they have the same set of optimal solutions.) Consider the following nonlinear program (NLP2):
\vspace{-1.5mm}
\begin{eqnarray}
\label{nlp2}
\max_{r_{k}} && 
\sum_{k=1}^{M} \sum_{i,j} r_{k}  \lambda_{dsk}(r_{k})  j, \\
\textrm{s.t.} && \sum_{k=1}^{M} \sum_{d,s}  \lambda_{dsk}(r_{k}) s \le ( 1-\epsilon) C, \nonumber \\
&& 0 \le r_{k} \le 1, \qquad \forall k. \nonumber
\end{eqnarray}
It can be readily verified that as long as $\lambda_{dsk}(r_{k})$ is nonnegative (and decreasing) it is always optimal to have nonnegative prices, so the nonnegativity constraints can be dropped.

\begin{theorem}
\label{equiv}
The programs (NLP1) and (NLP2) are equivalent.
\end{theorem}

\begin{proof}[Proof of Theorem \ref{equiv}.]
First, we show that for each solution $\{r_{k}\}$ of (NLP2), we can construct a solution of (NLP1) with the same objective value. Specifically, consider a solution $\{r'_{k}, \alpha'_{dsk}\}$ such that $r'_{k} = r_{k}$ and $\alpha'_{dsk}=1$ if and only if $\sum_{d,s}  \lambda_{dsk}(r_{k}) s >0$. It can be verified that the resulting solution is feasible for (NLP1) and has the same objective value.

Next, we show how to map optimal solution $\{r^{*}_{k}, \alpha^{*}_{dsk}\}$ of (NLP1) to a feasible solution of (NLP2) with the same objective function. For each $k=1,\ldots,M'-1$, set $r_{k}=r^{*}_{k}$, and for each $k=M'+1,\ldots,M$, set $r_{k}=r_{\infty}$. It is clear that, for each $k \ne M'$, the resulting contributions to the objective value and constraint in (NLP2) are the same as in (NLP1). Consider now the possible fractional value $\alpha^{*}_{M'}$ for class $M'$. The respective contribution of class $M'$ to the objective value is $\sum_{d,s} r^{*}_{M'} \alpha^{*}_{dsM'} \lambda_{ijM'}(r^{*}_{M'}) j$. Similarly, the contribution to constraint in (\ref{mlp}) is $\sum_{d,s} \alpha^{*}_{dsM'} \lambda_{dsM'}(r^{*}_{M'}) s$. Thus, it is sufficient to show that there exists a price $r_{M'}$ such that $\sum_{d,s} r_{M'} \lambda_{dsM'}(r_{M'}) s \ge \sum_{d,s} r^{*}_{M'} \alpha^{*}_{dsM'} \lambda_{dsM'}(r^{*}_{M'}) s$ and $\sum_{d,s} \lambda_{dsM'}(r_{M'}) s \le \sum_{d,s}  \alpha^{*}_{dsM'} \lambda_{dsM'}(r^{*}_{M'}) s$.

Since $\sum_{d,s} r^{*}_{M'} \lambda_{dsM'}(r^{*}_{M'}) s \ge \sum_{d,s} r^{*}_{M'} \alpha^{*}_{dsM'} \lambda_{dsM'}(r^{*}_{M'}) s$, by the properties of 
$\lambda_{dsM'}(r_{M'})$, we know that there exists $\bar{r} \in [r_{M'}, r_{\infty})$ such that $\sum_{d,s} \bar{r}  \lambda_{dsM'}(\bar{r}) s = \sum_{d,s} r^{*}_{M'} \alpha^{*}_{dsM'} \lambda_{dsM'}(r^{*}_{M'}) j$. Note that 
$\bar{r} \ge r^{*}_{M'}$, and therefore we obtain $\sum_{d,s} r^{*}_{M'} \lambda_{dsM'}(\bar{r}) s  \le \sum_{d,s} \bar{r}  \lambda_{dsM'}(\bar{r}) s = \sum_{d,s} r^{*}_{M'} \alpha^{*}_{dsM'} \lambda_{dsM'}(r^{*}_{M'}) s$. We conclude that 
$\sum_{d,s} \lambda_{dsM'}(\bar{r}) s  \le \sum_{d,s} \alpha^{*}_{dsM'} \lambda_{dsM'}(r^{*}_{M'}) s$, which completes the proof.  
\end{proof}

Lemma \ref{equiv} implies that we can solve (NLP2) instead of solving (NLP1). However, (NLP2) is computationally more tractable and can be solved relatively easy in many scenarios. Specifically, Lagrangify (dualize) the constraint in (NLP2) with some Lagrange multiplier $\Theta$ and consider the unconstraint problem $\max_{r_{k} \in [\Theta, r_{\infty})} \sum_{1\le k \le M} \sum_{d,s} (r_{k} -\Theta) \lambda_{dsk}(r_{k}) s$, which is separable in $r_{1}, \ldots, r_{M'}$. In fact, one aims to find the minimal $\Theta$ for which the resulting solution satisfies the constraint in (NLP2). This can be done by applying bi-section search on the interval $[0,p_{\infty}]$. The complexity of this procedure depends on the complexity of maximizing $\sum_{1\le k \le M} \sum_{d,s} (r_{k} -\Theta) \lambda_{dsk}(r_{k}) s$ for each $1\le k \le M$. It is not hard to check that there are at least two tractable cases: (i) $\lambda_{dsk}(r_{k})$ is a concave function on $[0,r_{\infty})$, for each $1\le k \le M$; (ii) $\lambda_{dsk}(r_{k})$ is convex, but $r_{k}\lambda_{dsk}(r_{k})$ is concave function on $[0,r_{\infty})$, for each $1\le k \le M$.

\section{Numerical Experiments} \label{sec_NE}
In this section, we examine the empirical performance of the ICSP in comparison with the optimal solutions via dynamic programming. It should be noted that the dynamic program is computationally expensive and we have to restrict ourselves to 3 classes of customer arrivals. In contrast, solving the ICSP solutions is extremely efficient and the results show that the performance gap is very small under a large set of instances. The numerical studies also suggest that the ICSP performs close to optimal in light and medium traffic as well. For all tested instances, the average CPU time per test instance on a Pentium 2.70GHz PC is 4.5s, whereas the dynamic programming algorithm is extremely slow per test instance.
 
\paragraph{Dynamic Programming Formulation.}
We describe the dynamic program used to compute the optimal policies. Incoming requests arrive over a time horizon of $T$ discrete periods. The state variable is the remaining capacity for each period $\mathbf c=(c_1,c_2,\ldots,c_T)^T$. Initially, $c=(c_0,c_0,\ldots,c_0)^T$ where $c_{0}$ is the maximum capacity of the homogeneous resource pool. The reward for different classes of customers $\mathbf r=(r_1,r_2,\ldots,r_M)^T$ where $M$ is number of classes. Their arriving probability is $\mathbf \lambda=(\lambda_1,\lambda_2,\ldots,\lambda_M)^T$ according to Poisson processes. The booking slot or the advanced reserved slot is defined as a Kronecker delta vector $\mathbf w=\delta_{ab}$, that is, $w(i)=1, \forall i \in [a,b]$ and $w(i)=0$ otherwise. The vector $\mathbf w$ represents the arriving request by flagging each of its requested service period $1$ between the starting service time $a$ and ending service time $b$.

The value function is defined as follow. At the beginning of time t, with remaining capacity $\mathbf c$, the optimal expected cumulative reward one could get is $\mathbb{E}V_t^c=\sup_{\pi\in\Pi}\{\mathbb{E}(V_t^c|\pi)\}$. It turns out that the optimal decision follows a threshold rule:
\begin{equation}
 \pi^*(t,\mathbf c,\mathbf w, \mathbf r)=\begin{cases}
\mbox{accept, if }\mathbf r \mathbf e^T \mathbf w+ \mathbb{E}V_{t+1}^{c-w}> \mathbb{E}V_{t+1}^c \mbox{ and } \mathbf w \le \mathbf c, \\
\mbox{reject, if }\mathbf r \mathbf e^T \mathbf w+\mathbb{E}V_{t+1}^{c-w}\le \mathbb{E} V_{t+1}^c \mbox{ or }\mathbf w\le \mathbf c \mbox{ doesn't hold},\\
  \end{cases}\nonumber
\end{equation}
where $\mathbf e$ is a column vector of all ones, and $\mathbf w \le \mathbf c$ component-wise means that the remaining capacities can serve the incoming request. The bellman update is then
\begin{eqnarray*}
\mathbb{E}V_t^c &=& \mathbb{P}(\mathbf r \mathbf e^T \mathbf w+ \mathbb{E}V_{t+1}^{c-w}>\mathbb{E}V_{t+1}^c, \mathbf w\le \mathbf c)\cdot \mathbb{E}(\mathbf r \mathbf e^T \mathbf w
+ \mathbb{E}V_{t+1}^{c-w} \mid \mathbf w\le \mathbf c, \mathbf r \mathbf e^T \mathbf w + \mathbb{E}V_{t+1}^{c-w}> \mathbb{E}V_{t+1}^c) \\
&& + (1-\mathbb{P}(\mathbf r \mathbf e^T \mathbf w+\mathbb{E}V_{t+1}^{c-w}>\mathbb{E}V_{t+1}^c, \mathbf w\le \mathbf c)) \cdot \mathbb{E}V_{t+1}^c,
\end{eqnarray*}
with boundary condition $\mathbb{E}V_t^c=0,\forall \mathbf c$ and $t>T$. No arrival is considered rejected arbitrarily. We then define critical reward:
\begin{equation}
 R_t^c(w)=\begin{cases}
\mathbb{E}V_{t+1}^c-\mathbb{E}V_{t+1}^{c-w}, \mathbf w\le \mathbf c \\
\infty, \mbox{otherwise}\\
  \end{cases}\nonumber
\end{equation}
Notice $ R_t^c(w)$ is a nondecreasing function of $\mathbf w$ for all $t$ and $\mathbf c$. The optimal decision rule can re-written as follows,
\begin{equation}
\pi^*(t,c,w,r)=\begin{cases}
\mbox{accept, if }\mathbf r \mathbf e^T \mathbf w>R_t^c(w),\\
\mbox{reject, if }\mathbf r \mathbf e^T \mathbf w\le R_t^c(w).\\
  \end{cases}\nonumber
\end{equation}
The resulting dynamic program suffers from the curse of dimensionality due to the multidimensional state vector. 

\paragraph{Numerical Results.} 
We consider three classes of customers. The default parameters (the base case) are as follows,
\begin{eqnarray*}
&&\lambda= 30,\qquad (\lambda_{1}, \lambda_{2}, \lambda_{3}) = (0.1\lambda, 0.45\lambda, 0.45\lambda) \\
&& \mathbf{r} = (15, 10, 8), \qquad C = 60, \qquad T=120.
\end{eqnarray*}
The advanced reservation and service distributions are assumed to be discretized exponential or normal distributions with different coefficient of variations ($c.v. = 0.1, 0.25, 0.5$). We benchmark the empirical performance of ICSP (with $\epsilon = 0.001$)  against the optimal dynamic programming solution. The extensive numerical results show that the ICSP performs within a few percentages of the optimal revenue. We define the performance measure of the ICSP as
$$
\text{err} = \left| \frac{\mathcal{R}(ICSP) - \mathcal{R}(OPT)}{\mathcal{R}(OPT)} \right|
$$ 
Table \ref{change_service} shows the numerical results with sensitivities on the mean advanced reservation times $\nu$ or the mean service times $\mu$. Table \ref{change_capacity} shows the numerical results where the total capacity is scaled or unscaled with the arrival rate. It should be noted that the performance gap shrinks when the capacity $C$ increases, which is consistent with our analytical analysis. Finally, Table \ref{change_reward_prob} shows the numerical results with sensitivities on the reward vector or the arrival probabilities. The results indicate that the performance gap is robust with respect to these input parameters. The typical performance of the ICSP is within $7\%$ of optimal ((average error of less than 4\%)  in all test instances.

\begin{table}[htbp]
  \centering
    \begin{tabular}{cccc|cccc}
    \hline \hline
    $\nu$    & DP    & err   & Accept  & $\mu$    & DP    & err   & Accept  \\
    
          & ($10^6$) &       & Prob. &       & ($10^6$) &       & Prob. \\ \hline
    \multicolumn{4}{c|}{$\mu= 8$ } &  \multicolumn{4}{c}{$\nu = 8$ }   \\ \hline
    2     & 6.99  & 1.78\% & (1, 0.44, 0) & 2     & 5.51  & 1.02\% & (1, 1, 0.56) \\
    4     & 6.79  & 1.26\% & (1, 0.44, 0) & 4     & 5.63  & 2.60\% & (1, 0.91, 0) \\
    6     & 6.59  & 2.37\% & (1, 0.44, 0) & 6     & 6.13  & 3.11\% & (1, 0.58, 0) \\
    8     & 6.35  & 2.81\% & (1, 0.44, 0) & 8     & 6.35  & 2.81\% & (1, 0.44, 0) \\
    10    & 6.21  & 3.73\% & (1, 0.44, 0) & 10    & 6.67  & 3.10\% & (1, 0.36, 0) \\
    12    & 6.19  & 4.89\% & (1, 0.44, 0) & 12    & 6.79  & 3.87\% & (1, 0.32, 0) \\
    \hline \hline
    \end{tabular}%
		\caption{Performance of ICSP: changing the mean advanced reservation times $\nu$ or the mean service times $\mu$.}
		\label{change_service}
\end{table}%

\begin{table}[htbp]
  \centering
    \begin{tabular}{cccc|cccc}
    \hline \hline
    C     & DP    & err   & Accept  & C     & DP    & err   & Accept  \\
          & ($10^6$) &       & Prob. &       & ($10^6$) &       & Prob. \\ \hline
    \multicolumn{4}{c|}{$\lambda = 30$} & \multicolumn{4}{c}{$C=\rho^{*}$}   \\ \hline
    30    & 3.66  & 6.51\% & (1, 0.11, 0) & 22    & 2.16  & 6.43\% & (1, 0.61, 0) \\
    40    & 4.51  & 6.99\% & (1, 0.22, 0) & 44    & 4.53  & 3.76\% & (1, 0.61, 0) \\
    50    & 5.45  & 3.54\% & (1, 0.33, 0) & 66    & 6.91  & 2.27\% & (1, 0.61, 0) \\
    60    & 6.35  & 2.83\% & (1, 0.44, 0) & 88    & 9.29  & 2.38\% & (1, 0.61, 0) \\
    70    & 7.39  & 2.42\% & (1, 0.55, 0) & 110   & 11.8  & 1.20\% & (1, 0.61, 0) \\
    80    & 8.41  & 1.27\% & (1, 0.66, 0) & 132   & 14.1  & 0.38\% & (1, 0.61, 0) \\
    \hline \hline
    \end{tabular}%
 	\caption{Performance of ICSP: changing the total capacity $C$ (unscaled and scaled).}
	\label{change_capacity}
\end{table}%

\begin{table}[htbp]
  \centering
    \begin{tabular}{cccc|cccc}
    \hline \hline
    $\mathbf{r}$     & DP    & err   & Accept  & Arrival & DP    & err   & Accept  \\
          &  ($10^6$) &       & Prob. & Prob. &  ($10^6$) &       & Prob. \\ \hline
    \multicolumn{4}{c|}{Arrival Prob. = (0.1, 0.45, 0.45)} & \multicolumn{4}{c}{$\mathbf{r}$ = (15, 10 , 8)} \\ \hline
    (20, 10, 1) & 3.66  & 1.80\% & (1, 0.11, 0) & (0.1, 0.3, 0.6) & 6.42  & 1.97\% & (1, 0.66, 0) \\
    (20, 10, 8) & 4.51  & 1.86\% & (1, 0.22, 0) & (0.1, 0.45, 0.45) & 6.35  & 2.81\% & (1, 0.44, 0) \\
    (80, 10, 5) & 5.45  & 1.59\% & (1, 0.33, 0) & (0.05, 0.87, 0.08) & 6.04  & 3.60\% & (1, 0.28, 0) \\
    (100, 10, 5) & 6.35  & 2.35\% & (1, 0.44, 0) & (0.1, 0.82, 0.08) & 6.51  & 3.77\% & (1, 0.24, 0) \\
    (150, 10, 5) & 7.39  & 0.63\% & (1, 0.55, 0) & (0.3, 0.62, 0.08) & 8.27  & 2.56\% & (1, 0.00, 0) \\
    \hline \hline
    \end{tabular}%
  	\caption{Performance of ICSP: changing the reward vector $\mathbf{r}$ or the arrival probabilities.}
	\label{change_reward_prob}
\end{table}%

\section{Conclusion} \label{sec_con}
To close this paper, we would like to point out three potential research directions. (a) One may consider a dynamic pricing model and derive similar policies. Assume that there is a single-class time-homogenous Poisson arrival process with rate $\lambda$. Each customer's reservation and service-time are drawn from $D$ and $S$, respectively. The system offers a price from a fixed price menu $[r_{1},\ldots, r_{n}]$ to an arriving customer with $d$ and $s$, depending on the current state. The state is characterized by the booking profile, $d$, and $s$. Moreover, one can introduce a reservation price distribution denoted by $R$. The customer only accepts the offer if the price offered falls below the reservation price. (b) \cite{GH13} considered a competition network model of perishable resources. One could also study a counterpart model of reusable resources. Each firm has a fixed capacity of reusable resources and competes in setting prices to sell them. Assuming deterministic customer arrival rates, one can potentially show that any equilibrium strategy has a simple structure, and then show that there exists a similar asymptotic equilibrium strategy in a stochastic version where the arrival rates and the capacity are scaled together to infinity. (c) To capture seasonality of demands, one could also consider the model studied in this paper with non-homogeneous Poisson arrivals. However, this may require entirely different methodologies.

\section*{Acknowledgment}
The research of Retsef Levi was partially supported by NSF grants DMS-0732175 and CMMI-0846554 (CAREER Award), AFOSR awards FA9950-11-1-0150 and FA9550-08-1-0369, an SMA grant and the Buschbaum Research Fund of MIT. The research of Cong Shi is partially supported by NSF grants CMMI-1362619 and CMMI-1451078. The authors would like to thank Ana Radovanovic (Google) and Yuan Zhong (Columbia) for their valuable comments and suggestions. The authors thank Sarah Liu for her help in the extensive numerical studies.

\bibliographystyle{ormsv080}
\bibliography{references}

\section*{Appendix}
\subsection*{Proof of Ergodicity}
In this section, we prove the existence and uniqueness of stationary distribution for the Markov chain induced by the improved class selection policy (ICSP). Let requests for resources from a common resource pool of capacity $C\le \infty$ arrive at time points $\{\tau_{n}, -\infty < n < \infty \}$. By observing the system at the moments of request arrivals, we define a discrete time process $I_{n} \triangleq (N_{n}^{(C)}, L_{i}, D_{i}, S_{i}, i=1,2,\ldots N_{n}^{(C)})$ where $N_{n}^{(C)}$ is the number of active (reserved) requests in the system at the moment of $n$th arrival $\tau_{n}$, $L_{i}$ is the elapsed time from the arrival of the $i^{th}$ request to $\tau_{n}$, $D_{i}$ and $S_{i}$ represent the reservation time (between arrival and actual service) and service time of the $i^{th}$ request, respectively. Note that $L_{i} \le D_{i} + S_{i}$ for $i=1,2,\ldots N_{n}^{(C)}$. The discrete-time Markov chain $I_{n}$ describes the entire booking profile at the moment of $n$th arrival $\tau_{n}$. We use a discrete version of Theorem 1 in \cite{sevastyanov} to prove the existence of a unique stationary distribution for $\{I_{n}\}$, which we state next for completeness.

\begin{theorem}
\label{appthm}
A Markov chain homogeneous in time has a unique stationary distribution which is ergodic if, for any $\epsilon >0$, there exists a 
measurable set $S$, a probability distribution $R$ on $\Omega$, and $n_{1}>0$, $k>0$, $K>0$ such that
\begin{itemize}
\item $kR(A) < P_{n_{1}}(x,A)$ for all points $x \in H$ and measurable sets $A \subset H$; for any initial distribution $P_{0}$ there exists $n_{0}$ such that for any $n \ge n_{0}$,
\item $P_{n}(H) \ge 1- \epsilon$,
\item $P_{n}(A) \le KR(A) + \epsilon$ for all measurable sets $A \subset H$.
\end{itemize}
\end{theorem}

\begin{proof}[Proof of Ergodicity.]
The proof follows similar arguments as in \cite{lura} and \cite{sevastyanov}. Define set $H(a,b,c,d)$ as 
\begin{equation}
H(a,b,c,d) \triangleq \left\{ N_{n}^{(C)} \le a, 0 \le L_{i} \le b, 0 \le D_{i} \le c, 0 \le S_{i} \le d \right\},
\end{equation}
for some positive finite constants $a,b,c,d$. Now we show that for any $\epsilon >0$, there exists $H(a,b,c,d) \in \Omega$, such that for any initial distribution $P_{0}$ there exists $n_{0}$ such that for all $n \ge n_{0}$,
\begin{equation}
P_{n}(H(a,b,c,d)) \ge 1- \epsilon.
\end{equation}
Note that 
\begin{eqnarray}
P_{n}(\bar{H}(a,b,c,d)) &\le&  \mathbb{P}\left[ N_{0,n}^{(C)} \ge a \right] 
+ \mathbb{P}\left[\bigcup_{i\in N_{0,n}^{(C)}} \left\{ L_{i} > b
\right\}, N_{0,n}^{(C)} \le a \right]\\
&& + \mathbb{P}\left[\bigcup_{i\in N_{0,n}^{(C)}} \left\{ D_{i} > c
\right\}, N_{0,n}^{(C)} \le a \right] +
\mathbb{P}\left[\bigcup_{i\in N_{0,n}^{(C)}} \left\{ S_{i} > d
\right\}, N_{0,n}^{(C)} \le a \right] \nonumber \\
&\le& \mathbb{P} \left[N_{a,n}^{(C)} + N_{0,n}^{0} > a \right] 
+ a \mathbb{P}\left[ L_{i} > b \right] + a \mathbb{P}\left[ D_{i} > c \right] + a \mathbb{P}\left[ S_{i} > d \right], \nonumber
\end{eqnarray}
where $N_{a,n}^{(C)}$ represents the number of active requests at $\tau_{n}$ that originated from $n$ arrivals at $\tau_{0}, \ldots, \tau_{n-1}$, and the rest of active requests at $\tau_{n}$, $N_{0,n}^{0}= N_{0,n}^{(C)}-N_{a,n}^{(C)}$ are those that were active at the initial point $\tau_{0}$ and are still active in the system at the moment of $n$th arrival. Next, since
\begin{eqnarray}
\label{app1}
&& \mathbb{P}\left[N_{a,n}^{(C)} + N_{0,n}^{0} > a \right] \\
&\le& \mathbb{P}\left[N_{a,n}^{(C)} > \frac{a}{2} \right] + 
\mathbb{P}\left[ N_{0,n}^{0} > \frac{a}{2} \right] 
\le \mathbb{P}\left[N_{n}^{(\infty)} > \frac{a}{2} \right] + 
\mathbb{P}\left[ \sum_{i=1}^{N^{0}_{0,n}} 
\mathds{1}\left[D_{i}^{0} + S_{i}^{0} > \tau_{n}-\tau_{0}\right] > \frac{\psi}{2} \right] \nonumber \\
&\le& \mathbb{P}\left[N_{n}^{(\infty)} > \frac{a}{2} \right] + \sum_{m=0}^{\infty}\mathbb{P}\left[N^{0}_{0,n} =m \right] 
\mathbb{P}\left[ \sum_{i=1}^{m} \mathds{1}\left[D_{i}^{0} + S_{i}^{0} > (1-\epsilon_{1})n\mathbb{E}(\tau_{1}-\tau_{0})
\right] > \frac{a}{2} \right] \nonumber \\
&& +\mathbb{P}\left[  \tau_{n}-\tau_{0} <(1-\epsilon_{1})n\mathbb{E}(\tau_{1}-\tau_{0})) \right], \nonumber 
\end{eqnarray}
where $0 <\epsilon_{1}<1$ is an arbitrary constant and we used $N_{n}^{(\infty)} \ge N_{a,n}^{(C)}$ a.s. where $N_{n}^{(\infty)}$ is the active requests under infinite capacity system.

Next we prove that there exists $a = a_{0}$ large enough such that (\ref{app1}) is bounded by $\epsilon/4$. By virtue of Little's Law, we know that $\mathbb{E} N_{n}^{(\infty)} < \infty$ and therefore, uniformly for all $n>0$, 
\begin{equation}
\label{app2}
\lim_{a \rightarrow \infty} \mathbb{P}\left[N^{(\infty)}_{n} > \frac{a}{2} \right] \rightarrow 0.
\end{equation}
Next, note that $\mathds{1}\left[D_{i}^{0} + S_{i}^{0} > (1-\epsilon_{1})n\mathbb{E}(\tau_{1}-\tau_{0})\right]
\le \mathds{1}\left[D_{i}^{0} + S_{i}^{0} > (1-\epsilon_{1})\mathbb{E}(\tau_{1}-\tau_{0})\right]$ a.s., and that for any fixed $m$,
\begin{eqnarray}
\label{app3}
&&\mathbb{P}\left[\sum_{i}^{m}\mathds{1}\left[D_{i}^{0} + S_{i}^{0} > (1-\epsilon_{1})n\mathbb{E}(\tau_{1}-\tau_{0})\right] > \frac{a}{2}
\right] \\
&\le& \mathbb{P}\left[\sum_{i}^{m}\mathds{1}\left[D_{i}^{0} + S_{i}^{0} > (1-\epsilon_{1})\mathbb{E}(\tau_{1}-\tau_{0})\right] > \frac{a}{2}
\right] \downarrow 0 \qquad \text{as} \qquad a \rightarrow \infty, \nonumber 
\end{eqnarray}
which by the monotone convergence theorem implies that, uniformly for all $n>0$,
\begin{eqnarray}
\lim_{a \rightarrow \infty} \sum_{m=0}^{\infty}\mathbb{P}\left[N^{0}_{0,n} = m \right] \mathbb{P}\left[ \sum_{i=1}^{m} \mathds{1}\left[D_{i}^{0} + S_{i}^{0} > (1-\epsilon_{1})n\mathbb{E}(\tau_{1}-\tau_{0}) \right] > \frac{a}{2} \right] =0.
\end{eqnarray}
Finally, by the Weak Law of Large Numbers, for all $n$ large enough,
\begin{equation}
\label{app4}
\mathbb{P}\left[  \tau_{n}-\tau_{0} <(1-\epsilon_{1})n\mathbb{E}(\tau_{1}-\tau_{0})) \right] \le \epsilon / 12.
\end{equation}
Thus, by (\ref{app2}) and (\ref{app3}), for an arbitrary $0 <\epsilon < 1$, there exists $n_{0} < \infty$ and $a_{0} < \infty$ large enough such that for all $n\ge n_{0}$, 
\begin{equation}
\label{app5}
\mathbb{P}\left[N^{(\infty)}_{n} > \frac{a_{0}}{2} \right] \le \epsilon / 12,
\end{equation}
and
\begin{equation}
\label{app6}
\sum_{m=0}^{\infty}\mathbb{P}\left[N^{0}_{0,n} = m \right] \mathbb{P}\left[ \sum_{i=1}^{m} \mathds{1}\left[S_{i}^{0} + D_{i}^{0} > (1-\epsilon_{1})n\mathbb{E}(\tau_{1}-\tau_{0}) \right] > \frac{a_{0}}{2} \right] \le \epsilon / 12.
\end{equation}
Now since $\mathbb{E}L_{i} < \infty$, $\mathbb{E}D_{i} < \infty$ and $\mathbb{E}S_{i} < \infty$, 
there exists $b_{0}$, $c_{0}$ and $d_{0}$ such that 
\begin{equation}
\label{app7}
\mathbb{P}\left[L_{i} > b_{0} \right] \le \frac{\epsilon}{4 a_{0}}, \qquad
\mathbb{P}\left[D_{i} > c_{0} \right] \le \frac{\epsilon}{4 a_{0}}, \qquad
\mathbb{P}\left[S_{i} > d_{0} \right] \le \frac{\epsilon}{4 a_{0}}.
\end{equation}
Thus, by (\ref{app4}), (\ref{app5}), (\ref{app6}) and (\ref{app7}), we have
\begin{equation}
P_{n}(\bar{H}(a,b,c,d)) \le \epsilon \qquad \Rightarrow \qquad P_{n}(H(a,b,c,d)) \ge 1 - \epsilon.
\end{equation}

Next, we show that there exists $n_{1}>0$ and $k>0$ such that for all points $x \in H(a_{0},b_{0},c_{0},d_{0})$ and measurable sets 
$A \in H(a_{0},b_{0},c_{0},d_{0})$, the following inequality holds
\begin{equation}
P_{n_{1}}(x,A) \ge k R(A).
\end{equation}
Let $F_{V}(v)$ denote a cumulative distribution function of a random duration $V$, i.e. $\mathbb{P}\left[V \le v \right]$.  Next, for any $n_{1}$,
\begin{equation}
P_{n_{1}}(x,A) \ge P_{1}(x,\omega_{0})P_{n_{2}}(\omega_{0},A),
\end{equation}
where $n_{2} = n_{1}-1$. Let $x=(m,l_{1},\ldots,l_{m},d_{1},\ldots,d_{m},s_{1},\ldots,s_{m}) \in H(a_{0},b_{0},c_{0},d_{0})$. Then,
\begin{eqnarray}
P_{1}(x,\omega_{0}) 
&\ge& \mathbb{P}\left[ \tau_{1} - \tau_{0} \ge \Delta, \text{all } m \text{ requests depart in } (\tau_{0},\tau_{1})\right] \\
&\ge& \mathbb{P}\left[ \tau_{1} - \tau_{0} > c_{0}+d_{0}\right] = 1-F_{a}(c_{0}+d_{0}) = e^{-\lambda(c_{0}+d_{0})}, \nonumber
\end{eqnarray}
where $F_{a}(u)$ represents cumulative inter-arrival distribution of a renewal process $\{\tau_{n}\}$, i.e. 
$$F_{a}(u)
=\mathbb{P}\left[\tau_{1}-\tau_{0} \le u \right]$$. 
Next, we derive a lower bound for $P_{n_{2}}(\omega_{0},A)$ for some $n_{2}$ large enough such that 
\begin{eqnarray}
\label{app8}
\mathbb{P}\left[ \tau_{n_{2}} - \tau_{0} \ge c_{0}+d_{0} \right] \ge 1-\frac{\epsilon}{2}.
\end{eqnarray}
Note that the condition imposed on $n_{2}$ is possible due to the Weak Law of Large Numbers, since for any $\epsilon > 0$ and all $n_{2}$ large enough with $c_{0}+d_{0}<(1-\epsilon)\mathbb{E}(\tau_{n_{2}}-\tau_{0})$,
\begin{eqnarray}
\mathbb{P}\left[ \tau_{n_{2}} - \tau_{0} \ge c_{0}+d_{0} \right] 
\ge \mathbb{P}\left[ \tau_{n_{2}} - \tau_{0} \ge  (1-\epsilon)\mathbb{E}(\tau_{n_{2}}-\tau_{0}) \right] \ge 1-\frac{\epsilon}{2}.
\end{eqnarray}
Next, pick any $x' =(m',l'_{1}, \ldots l'_{m'},d'_{1}, \ldots d'_{m'}, s'_{1}, \ldots, s'_{m'}) \in A$. Define $x' + dx' \triangleq 
(m',l'_{1}+dl'_{1}, \ldots l'_{m'}+dl'_{m'},d'_{1}+dd'_{1}, \ldots d'_{m'}+dd'_{m'}, s'_{1}+ds'_{1}, \ldots, s'_{m'}+ds'_{m'})$ where $dl'_{1}, \ldots, dl'_{m'},dd'_{1}, \ldots, dd'_{m'}, ds'_{1}, \ldots, ds'_{m'}$ are infinitesimal elements. Then the transition probability into state $(x',x'+dx')$ starting from $\omega_{0}$ can be lower bounded by the probability of the event that there are exactly $m'$ arrivals between $\tau_{1}$ and $\tau_{n_{2}}$ whose arrivals times are determined by $(\tau_{n_{2}} - l'_{i} -dl'_{i},\tau_{n_{2}} - l'_{i})$ for $i=1,\ldots,m'$, and none of these $m'$ arrivals concluded at time $\tau_{n_{2}}$ and there were no other arrivals. Therefore,
\begin{eqnarray}
\mathbb{P}_{n_{2}}(\omega_{0}, (x'+dx' )) \ge e^{-\lambda \tau_{n_{2}}} 
\frac{\lambda^{m'}}{m'!} \prod_{i=1}^{m'}\left[1-F_{D+S,D,S}(l_{i},d_{i},s_{i})\right],
\end{eqnarray}
where $F_{D+S,D,S}(\cdot)$ is the joint cumulative probability mass function. Now define probability distribution
\begin{eqnarray}
R(A) \triangleq \nu \int_{x' \in A} \frac{\lambda^{m'}}{m'!} \prod_{i=1}^{m'}\left[1-F_{D+S,D,S}(l_{i},d_{i},s_{i})\right],
\end{eqnarray}
where $\nu$ is a normalization constant. Thus, we have 
\begin{eqnarray}
P_{n_{2}+1}(x,A) \ge  e^{-\lambda(c_{0}+d_{0}+\tau_{n_{2}})} \nu^{-1} R(A).
\end{eqnarray}
Finally, it is left to show that there exists $K>0$ such that for every initial distribution $P_{0}$, for all $n$ large and for any measurable set $A\subset S(a_{0},b_{0},c_{0},d_{0})$,
\begin{eqnarray}
P_{n}(A) \le  K R(A) + \epsilon.
\end{eqnarray}
By (\ref{app8}), for all $n \ge n_{2}$,
\begin{eqnarray}
P_{n}(A) &\le& \mathbb{P} \left[ H_{n} \in A, \tau_{n}-\tau_{0} > c_{0} + d_{0} \right]
+ \mathbb{P} \left[ \tau_{n}-\tau_{0} \le c_{0} + d_{0} \right] \\
&\le& \mathbb{P}\left[ H_{n} \in A, \tau_{n}-\tau_{0} > c_{0} + d_{0} \right] + \frac{\epsilon}{2} \nonumber \\
&\le&   \int_{x' \in A} \left\{ \frac{\lambda^{m'}}{m'!} \prod_{i=1}^{m'}\left[1-F_{D+S,D,S}(l_{i},d_{i},s_{i})\right] 
\right\} + \epsilon \nonumber \\
&\le&   \nu^{-1}R(A) + \epsilon. \nonumber 
\end{eqnarray}
 
We have verified the conditions stated in Theorem \ref{appthm} and thus the process $\{H_{n}\}$ has a unique stationary distribution as well implying the existence of the stationary blocking probability. 
\end{proof}

\subsection*{Proofs of Technical Lemmas and Propositions}

\begin{proof}[Proof of Lemma \ref{firstl}.]
If $r \le 0$, we focus on the interval $(\left\lceil r\right\rceil-1, \left\lceil r\right\rceil]$ and its preceding interval $(\left\lceil r\right\rceil-2, \left\lceil r\right\rceil - 1]$. The arrival process in $(\left\lceil r\right\rceil-2, \left\lceil r\right\rceil - 1]$ follows a Poisson process with rate $\lambda$. Each arrival has $\gamma$ probability of starting services immediately in $(\left\lceil r\right\rceil-2, \left\lceil r\right\rceil - 1]$, and $1-\gamma$ probability of starting services in $1$ unit of time in $(\left\lceil r\right\rceil-1, \left\lceil r\right\rceil]$. By the Poisson splitting argument, the pre-arrivals in $(\left\lceil r\right\rceil-1, \left\lceil r\right\rceil]$ follow a Poisson process with rate $(1-\gamma)\lambda$. Using a similar argument, we conclude that the pre-arrival process in $(\left\lceil r\right\rceil-1, \left\lceil r\right\rceil]$ induced by customers arriving to the system in $(\left\lceil r\right\rceil-1, \left\lceil r\right\rceil]$ follows a Poisson process with rate $\gamma\lambda$. Note that these two processes are independent of each other since they are generated by customers arriving in disjoint intervals. Now merge these two pre-arrival processes, and the resulting pre-arrival process in $(\left\lceil r\right\rceil-1, \left\lceil r\right\rceil]$ follows a Poisson process with rate $(1-\gamma) \lambda + \gamma\lambda =\lambda$. 

If $0< r \le 1$, focus on the interval $(0,1]$ and its preceding interval $(-1,0]$. By an argument similar to the above, there is a Poisson process of pre-arrivals with rate $(1-\gamma)\lambda$ induced by customers arriving in $(-1,0]$. There is also a Poisson process with rate $\gamma \lambda$ induced by customers arriving in $(0,1]$. However, the latter process consists of post-arrivals. Thus, the resulting pre-arrivals at time $0$ over $(0, 1]$ follow a Poisson process with rate $(1-\gamma) \lambda$.

Finally, since the maximum reservation time is $1$, it is impossible for customers arriving prior to $0$ to start service at any time greater than $1$. Thus, the rate of pre-arrivals from $1$ onwards is $0$. This completes the proof. 
\end{proof}

\begin{proof}[Proof of Lemma \ref{sl1}.]
Suppose that a customer arrives at time $0$ in steady-state and requests the service to commence immediately ($D=0$), i.e., requesting the service interval $(0,1]$. Focus solely on the pre-arrivals as seen from $0$. By Lemma \ref{sst}, the pre-arrivals over the time interval $(-1,0]$ follow a Poisson process with rate $\lambda$, denoted by $N_{0}$. However, this implies that, over the time interval $(0,1]$, the customers depart the system following a Poisson process with rate $\lambda$ (a shift of $N_{0}$ by $1$ unit of time). Let $\tilde{N}_{0}$ be the mirror image of the departure process induced by $N_{0}$ over $(0,1]$. By Lemma \ref{sst}, we also know that the pre-arrivals over $(0,1]$ follow a Poisson process with rate $(1-\gamma)\lambda$. We denote this pre-arrival process by $N_{1}$. Figure \ref{adv3} shows the pre-arrival and departure processes.

Consider now the number of customers in the system at some time $r$. These fall exactly into one of the two types; customers that started service over $(0,r]$ and customers that started service over $(r-1,0]$ and will depart over $(r,1]$. It follows that the number of customers in service at time $r \in (0,1]$ can be expressed as $\tilde{N}_{0}(1-r) + N_{1}(r)$. Specifically, in time $r$ the number of departures over $(r,1]$ (equal to $\tilde{N}_{0}(1-r)$) captures customers starting service before $0$, and still in the system at time $r$. In addition, the number of pre-arrivals over $(0,r]$ (equal to $N_{1}(r)$) captures customers arriving before $0$, starting service over $(0,r]$ and still being served in time $r$. The sum of the two is exactly equal to the total number of customers in the system at time $r$. Note that by the Poisson splitting argument, it follows that $\tilde{N}_{0}$ and $N_{1}$ are independent of each other. The virtual blocking probability is expressed in terms of the maximum of the sum of these two Poisson counting processes running towards each other (see Figure \ref{adv2}), i.e.,
$$
P_{0} \triangleq \mathbb{P}(B \mid D=0) =  \mathbb{P} \left(\max_{r \in [0,1]}
\left\{ \tilde{N}_{0}(1-r; \lambda) + N_{1}(r;(1-\gamma)\lambda) \right\} \ge C \right).
$$

Consider now the case that the arriving customer requests the service to commence in $D=1$ unit of time, i.e., the service will cover the interval $(1,2]$. The departure process in $(1,2]$ is a shift of the pre-arrival process $N_{1}$ in $(0,1]$ by $1$ unit of time, and its mirror image is denoted by $\tilde{N}_{1}$. Moreover, by Lemma \ref{sst}, the pre-arrival process $N_{2}$ in $(1,2]$ has rate $0$. Thus, we have
\begin{eqnarray*}
P_{1} \triangleq \mathbb{P}(B \mid D=1) = \mathbb{P} \left( \max_{r \in [0,1]} 
\left\{ \tilde{N}_{1}(1-r; (1-\gamma)\lambda) + N_{2}(r;0) \right\} \ge C \right)=\mathbb{P} \left(  
 \tilde{N}_{1}(1; (1-\gamma)\lambda) \ge C \right).
\end{eqnarray*}
This completes the proof. 
\end{proof}

\begin{proof}[Proof of Lemma \ref{decomposition}.]
Note again that for each $r \in [0,1]$, $\tilde{N}_{0}(1-r; \theta_{1}\lambda) + N_{1}(r;\lambda)$ is equal to the sum of the number of occurrences of $N_{0}$ over $(1-r,1]$ and the number of occurrences of $N_{1}$ over $[0,r)$. Also observe that the value of $X$ is obtained either at time $0$ or upon on occurrence of $N_{1}$. Now condition on $\mathcal{R}_{n} = \omega_{n}$ (a specific realization of the random walk $\mathcal{R}_{n}$), and consider the $l^{th}$ occurrence of $N_{1}$ ($l \in \{0,\ldots,n\}$), at time, say $r$. Then we have (see Figure \ref{adv2}),
\begin{eqnarray*}
&&\tilde{N}_{0}(1-r; \lambda)+N_{1}(r; \theta_{1}\lambda) 
= (\text{\# up-steps before and including } l + \text{\# down-steps after }l )\\
&=& (\text{\# up-steps before and including }l - \text{\# down-steps before and including }l)  \\
&& + (\text{\# down-steps before and including }l + \text{\# down-steps after }l) .
\end{eqnarray*}
The first term is exactly the location of the random walk after $l$ steps and the second expression is exactly $G_{n}$. Since $X$ is the maximum of the above sum over all arrivals $l=0,1,\ldots,n$, it follows that indeed $X_{n} \mid (\mathcal{R}_{n} = \omega_{n}) = (G_{n} + M_{n}) \mid (\mathcal{R}_{n} = \omega_{n})$, from which the result follows. 
\end{proof}

\begin{proof}[Proof of Lemma \ref{hitting}.]
Define the stopping time $\tau$ as follows, $\tau = \inf \left\{ t\ge 1: S_{t} \le -a \textrm{ or } S_{t}\ge b
\right\}.$ It is straightforward to check the following two conditions, 
\begin{eqnarray}
\mathbb{E}(\tau) \le \infty, \qquad\mathbb{E}(|E_{t+1} - E_{t}| \mid \mathcal{F}_{t}) \le 2, \qquad \forall t \in \tau.
\end{eqnarray}
The Wald's identity (see \cite{gallager})
\begin{eqnarray}
G_{n}(\theta) \triangleq \frac{e^{\theta S_{n}}}{[\phi(\theta)]^{n}}
\end{eqnarray}
is a martingale where the moment generating function $\phi(\theta) \triangleq 
\mathbb{E}(e^{\theta Y}) \ge 1$. First we compute $\hat{\theta}$ that solves the equation
$\mathbb{E}(e^{\hat{\theta} Y}) = 1$, i.e.,
\begin{eqnarray}
\mathbb{E}(e^{\hat{\theta} Y}) = pe^{\hat{\theta}} + qe^{\hat{\theta}} = 1
\qquad \Rightarrow \qquad e^{\hat{\theta}} = \frac{q}{p}.
\end{eqnarray}
By Optional Sampling Theorem (see \cite{gallager}),
\begin{eqnarray}
\mathbb{E}\left[  \frac{e^{\hat{\theta} S_{\tau}}}{[\phi(\hat{\theta})]^{\tau}}
\right] = \mathbb{E}\left[  e^{\hat{\theta} S_{\tau}}\right] = \mathbb{E}\left[  e^{\hat{\theta} S_{0}}\right] = 1.
\end{eqnarray} 
This leads to 
\begin{eqnarray}
\mathbb{P}(S_{\tau} \ge b) \underbrace{\mathbb{E}(e^{\hat{\theta} S_{\tau}} \mid S_{\tau} \ge b)}_{E_{b}} + (1-\mathbb{P}(S_{\tau} \ge b)) \underbrace{\mathbb{E}(e^{\hat{\theta} S_{\tau}} \mid S_{\tau} \le -a) }_{E_{a}} = 1.
\end{eqnarray}
Thus, we have
\begin{eqnarray}
\mathbb{P}(S_{\tau} \ge b) = \frac{1-E_{a}}{E_{b} - E_{a}}
=\frac{1-e^{-\hat{\theta}a}}{e^{\hat{\theta}b} - e^{-\hat{\theta}a}}= \frac{1- \left(\frac{q}{p}\right)^{-a}}{\left(\frac{q}{p}\right)^{b} - \left(\frac{q}{p}\right)^{-a}}.
\end{eqnarray}
Let $S_{\tau}^{a} \triangleq S_{\tau}$ be the stopping time location of the process. Let $B_{a}$ be the event that the random walk hits $b$ before $-a$. Observe that $\mathbb{P}(B_{a}) = \mathbb{P}(S_{\tau}^{a} \ge b)$ and also note that $B_{i} \subset B_{i+1}$ for all $i$. Define $B = \bigcup_{i=1}^{\infty} B_{i}$, i.e., there exists an $i$ that the random walk hits $b$ before $-i$. Therefore $\mathbb{P}(M_{\infty} \ge b) = \mathbb{P}(B)$. By properties of probability measures, we have 
\begin{eqnarray}
\mathbb{P}(M_{\infty} \ge b) &=& \mathbb{P}(\bigcup_{i=1}^{\infty} B_{i}) 
= \lim_{a \rightarrow \infty} \mathbb{P}(B_{a}) =\lim_{a \rightarrow \infty} \left(\frac{1- \left(\frac{q}{p}\right)^{-a}}{\left(\frac{q}{p}\right)^{b} - \left(\frac{q}{p}\right)^{-a}}\right) = \left( \frac{p}{q}\right)^{b}.
\end{eqnarray} 
This completes the proof. 
\end{proof}

\begin{proof}[Proof of Lemma \ref{ext1}.]
Lemma \ref{sstg} is a generalized version of Lemma \ref{sst}. For $r \le 0$, consider the time interval 
$(\left\lceil r \right \rceil  -1, \left\lceil r \right \rceil]$. By arguments similar to those used in Lemma \ref{sst}, for each $l\in [0,u]$, the interval $(\left\lceil r \right \rceil -1 -l, \left\lceil r \right \rceil-l]$ generates a stream of pre-arrivals over $(\left\lceil r \right \rceil  -1, \left\lceil r \right \rceil]$ that follow a Poisson process of rate $\gamma_{l}\lambda$. These processes are independent of each other and the overall merged process has rate $\lambda = \gamma_{0} \lambda + \gamma_{1} \lambda + \ldots + \gamma_{u} \lambda$. For $\left\lceil r \right \rceil = d$ for $d \in [1,u]$, then the pre-arrivals prior to $t$ over $(\left\lceil r \right \rceil-1, \left\lceil r \right \rceil]$ are induced by arriving customers over the intervals 
$(\left\lceil r \right \rceil  -l - 1, \left\lceil r \right \rceil - l]$, for $l \in [d,u]$, and the total rate is 
$\gamma_{d} \lambda + \gamma_{d+1} \lambda + \ldots + \gamma_{u} \lambda$. Note again that the rate $\gamma_{i} \lambda$ is induced from the Poisson arrival stream of customers over $(\left\lceil r \right \rceil - l - 1, \left\lceil r \right \rceil - l]$ who wish to start in $l$ units of time. Since we only consider pre-arrivals prior to $t$, the terms $\gamma_{d-1}\lambda, \gamma_{d-2}\lambda, \ldots, \gamma_{0}\lambda$ are missing. 
\end{proof}

\begin{proof}[Proof of Lemma \ref{gl1}.]
By Lemma \ref{ext1}, for each $d \in [0,u]$, the pre-arrival process $N_{d}$ over the interval $(d-1,d]$ follows a Poisson process with rate $\lambda_{d}=\lambda\left(1- \sum_{i=0}^{d-1}\gamma_{i}\right)$.
This implies that over the interval $(d,d+1]$, the customers depart the system following a Poisson process with rate $\lambda_{d}$ (a shift of $N_{d}$ by 1 unit of time). Let $\tilde{N}_{d}$ be the mirror image of the departure process induced by $N_{d}$ over 
$(d,d+1]$, and therefore $\tilde{N}_{d}$ has the same rate $\lambda_{d}$. The rest of arguments is identical to that of Lemma \ref{sl1}. 
\end{proof}

\begin{proof}[Proof of Proposition \ref{gl2}.]
First we assume that $\gamma_{0}>0$. By Lemma \ref{gl1}, we have that $\lambda_{0} > \lambda_{1}$ and $\lambda_{d} \ge \lambda_{d+1}$ for each $d \in [1,u]$. By Proposition \ref{primary}, it follows that,
\begin{eqnarray*}
\lim_{\lambda \rightarrow \infty} P_{0} = \frac{1}{2}; \qquad \lim_{\lambda \rightarrow \infty} P_{d} = 0, \qquad \textrm{for } d \in [1,u].
\end{eqnarray*}
Therefore $P_{i} \le \frac{1}{2}$ holds given that $\gamma_{0}>0$. In fact, we can relax the assumption of $\gamma_{0}>0$. If $\gamma_{0}=0$, it implies that over the interval $(0,1]$ (recall that the customer arrives at time $0$ in steady-state), the departure rate is equal to the pre-arrival rate, i.e., $\lambda_{0} = \lambda_{1} =\lambda$. Proposition \ref{primary} cannot be applied under this case. However, the fact that $\gamma_{0}=0$ implies that no arriving customers will start the service right away. Therefore, we do not have to consider the probability $P_{0}$ in the expression of $P$. Let the index $i = \min\{d: \gamma_{d} >0\}$. Then we have  $\gamma_{d}=0$ for 
$d \in [0,i-1]$, by the same argument, we can ignore the probabilities $P_{0}, \ldots, P_{i-1}$. Instead, again by Proposition \ref{primary}, we have  
\begin{eqnarray*}
\lim_{\lambda \rightarrow \infty} P_{i} = \frac{1}{2}; \qquad
\lim_{\lambda \rightarrow \infty} P_{d} = 0, \qquad  d \in [i+1,u].
\end{eqnarray*}
 This completes the proof. 
\end{proof}

\begin{proof}[Proof of Lemma \ref{rs}.]
For each set $s \in [1,v]$, and $d \in [0,u]$, the pre-arrivals prior to $t$ (i.e., $N_{d}^{s}$) over $(d-s,d-s+1]$ are induced by arriving customers over the intervals $(d-s-i,d-s-i+1]$ for $i= \max(0,d-s+1),\ldots,u$, and the total rate $\lambda_{d}^{s}$ is therefore
\begin{equation}
\gamma_{\max (0,d-s+1)}^{s}\lambda_{0}^{s}  + \ldots + \gamma_{u}^{s}\lambda_{0}^{s} = \lambda_{0}^{s} \left(1 - \sum_{i=0}^{d-s} \gamma_{i}^{s}\right) = \kappa_{s} \lambda \left(1 - \sum_{i=0}^{d-s} \gamma_{i}^{s}\right).
\end{equation}
Note again that the rate $\gamma_{i}^{s}$ is induced from the Poisson arrival stream of customers over $(d-s-i,d-s-i+1]$ who wish to start in $i$ units of time. It follows from the Poisson splitting argument that $N_{d}^{s}$ and $N_{d'}^{s'}$ are independent if $(d,s) \ne (d',s')$. This completes the proof. 
\end{proof}

\begin{proof}[Proof of Lemma \ref{ac}.]
The assumption $\gamma_{0}^{s} > 0$ for some $s\in [1,v]$ implies that in the interval $(0,1]$, the total departure rate is strictly greater than the total pre-arrival rate, i.e., $\sum_{s=1}^{v} \lambda_{0}^{s} > \sum_{s=1}^{v}\lambda_{s}^{s}$. For subsequent intervals $(d,d+1]$ for $d \ge 1$, we have $\sum_{s=1}^{v} \lambda_{d}^{s} \ge \sum_{s=1}^{v}\lambda_{d+s}^{s}$. Therefore the conditions of Proposition \ref{primary} are satisfied. Proposition \ref{primary} implies that 
\begin{eqnarray}
\label{fin}
&& \lim_{\lambda \rightarrow \infty} \mathbb{P}(A_{0}\ge C ) 
= \lim_{\lambda \rightarrow \infty} \mathbb{P}\left(\left(\sum_{s=2}^{v} \sum_{i=1}^{s-1} N_{i}^{s}(1;\lambda_{i}^{s}) + 
\sum_{s=1}^{v} \tilde{N}_{0}^{s}(1;\lambda_{1}^{s})
\right)\ge C\right)  \\
&=& \lim_{\lambda \rightarrow \infty} \mathbb{P}\left(\textrm{Poisson}\left(\sum_{s=2}^{v} \sum_{i=1}^{s-1} \lambda_{i}^{s} + 
\sum_{s=1}^{v} \lambda_{0}^{s} \right)\ge C \right) 
= \lim_{\lambda \rightarrow \infty} \mathbb{P}\left(\textrm{Poisson}\left(\sum_{s=1}^{v}  s \lambda^{s}_{0}\right)\ge C\right) \nonumber \\
&=& = \lim_{\lambda \rightarrow \infty} \mathbb{P}\left(\textrm{Poisson}\left(\sum_{s=1}^{v}  s \kappa_{s} \lambda \right)\ge C\right) 
= \lim_{\lambda \rightarrow \infty} \mathbb{P}(\textrm{Poisson}(\rho)\ge C) 
= \frac{1}{2}. \nonumber
\end{eqnarray}
The third equality of (\ref{fin}) follows from (\ref{rates}) in Lemma \ref{rs}. Similarly, we can show that $\lim_{\lambda \rightarrow \infty} \mathbb{P}(A_{d} \ge C ) = 0$ for $d \in [1,u]$. This completes the proof.  
\end{proof}

\begin{proof}[Proof of Proposition \ref{sec}.]
First we assume that $\gamma_{0}^{s} > 0$ for some $s \in [1,v]$. For each $s \in [1,v]$, by applying union bound and Lemma \ref{ac}, we have
\begin{eqnarray}
&&\lim_{\lambda \rightarrow \infty} P_{0}^{j} = \lim_{\lambda \rightarrow \infty} \mathbb{P}(\max (A_{0}, \ldots, A_{d-1})\ge C) 
\le \lim_{\lambda \rightarrow \infty} \mathbb{P}\left(\bigcup_{i=0}^{d-1} A_{i} \ge C \right) \\
&\le& \lim_{\lambda \rightarrow \infty} \sum_{i=0}^{d-1}\mathbb{P}(A_{i}\ge C) = \lim_{\lambda \rightarrow \infty} \mathbb{P}(A_{0}\ge C) 
= \frac{1}{2}. \nonumber
\end{eqnarray}
On the other hand, it is obvious that, for each $s \in [1,v]$, 
\begin{eqnarray}
\lim_{\lambda \rightarrow \infty} P_{0}^{s} = \lim_{\lambda \rightarrow \infty} \mathbb{P}(\max (A_{0}, \ldots, A_{s-1})\ge C) 
\ge \lim_{\lambda \rightarrow \infty} \mathbb{P}(A_{0}\ge C)  
= \frac{1}{2}, \nonumber
\end{eqnarray}
Similarly, $\lim_{\lambda \rightarrow \infty} P_{d}^{s}= 0 $ for each $d \in [1, u]$ and $s \in [1,v]$.  

We then drop the assumption that $\gamma_{0}^{s} > 0$ for some $s \in [1,v]$. Suppose now $\gamma_{0}^{s} = 0$ for all $s \in [1,v]$. This implies that no arriving customers at time $0$ will start the service over $(0,1]$, and hence we can ignore the blocking probability over this interval. Let the index
$$
i = \min \left\{d: \gamma_{d}^{s} > 0 \text{ for some } s \in [1,v] \right\}.
$$
Observe that no arriving customers at time $0$ will start the service over $(0,i]$. For the subsequent interval $(i,i+1]$, the total departure rate is strictly greater than the pre-arrival rate, i.e., $\sum_{s=1}^{v} \lambda_{i}^{s} > \sum_{s=1}^{v}\lambda_{i+s}^{j}$. Thus, it suffices to show that for each $s \in [1,v]$,
$\lim_{\lambda \rightarrow \infty} P^{s}_{i} = \lim_{\lambda \rightarrow \infty} \mathbb{P}(\max (A_{i}, \ldots, A_{i+s-1})\ge C) 
= 1/2$, and $\lim_{\lambda \rightarrow \infty} P^{s}_{d} = 0$ for each $d \in [i+1,u]$ and $s \in [1,v]$. The same arguments follow through.
\end{proof}

\begin{proof}[Proof of Lemma \ref{ub}.] Consider the following LP
\begin{eqnarray}
\label{olp}
\max_{\alpha_{dsk}^{\pi}} 
\sum_{k=1}^{M} \sum_{d,s} r_{k} \alpha_{dsk}^{\pi} \lambda_{dsk} s, \qquad
\textrm{s.t.} && \sum_{k=1}^{M} \sum_{d,s} \alpha_{dsk}^{\pi} \lambda_{dsk} s \le C,  \text{ } 0 \le \alpha_{dsk} \le 1, \text{ } \forall d,s,k.
\end{eqnarray}
Note the LP defined in (\ref{mlp}) differs from the LP defined in (\ref{olp}) by changing the right hand side of the capacity constraint to $(1-\epsilon)C$. Suppose the optimal solution of the LP defined in (\ref{olp}) is $\{\hat{\alpha}_{dsk}\}$. Now consider 
$\{\tilde{\alpha}_{dsk}\} = \{(1-\epsilon) \hat{\alpha}_{dsk}\}$. Since 
\begin{equation}
\sum_{k=1}^{M} \sum_{d,s} \tilde{\alpha}_{dsk} \lambda_{dsk} s = (1-\epsilon) \sum_{k=1}^{M} \sum_{d,s} \hat{\alpha}_{dsk} \lambda_{dsk} s \le (1-\epsilon)C,
\end{equation}
$\{\tilde{\alpha}_{dsk}\}$ is a feasible solution to (\ref{mlp}).  Then we have
\begin{equation}
\sum_{k=1}^{M} \sum_{d,s} r_{k} \alpha^{*}_{dsk} \lambda_{dsk} s \ge  \sum_{k=1}^{M} \sum_{d,s} r_{k} {\tilde{\alpha}}_{dsk} \lambda_{dsk} s  = (1-\epsilon) \sum_{k=1}^{M} \sum_{d,s} {\hat{\alpha}}_{dsk} \lambda_{dsk} s \ge (1-\epsilon)\mathcal{R}(OPT).
\end{equation}
The last inequality holds since the optimal objective value in (\ref{olp}) provides an upper bound on the optimal expected revenue rate, since the capacity constraint is enforced on expectation whereas in the original problem this capacity constraint has to hold, for each sample path. This completes the proof. 
\end{proof}

\end{document}